\documentclass[10pt]{amsart}
\usepackage{amssymb}
\usepackage{dsfont}
\usepackage{amscd}
\usepackage[mathscr]{euscript} 
\usepackage[all]{xy}
\usepackage{url}
\usepackage{stmaryrd}
\usepackage{comment}
\usepackage[retainorgcmds]{IEEEtrantools}
\usepackage{hyperref}
\title{Operators coming from ring schemes}
\author[J. GOGOLOK]{Jakub Gogolok$^{\dagger}$}
\thanks{$^{\dagger}$ Supported by the Narodowe Centrum Nauki grant no. 2018/31/B/ST1/00357.}
\address{$^{\spadesuit}$Instytut Matematyczny\\
Uniwersytet Wroc{\l}awski\\
pl. Grunwaldzki 2, 50-384
\\
Wroc{\l}aw\\
Poland}
\email{jakub.gogolok@math.uni.wroc.pl} \urladdr{http://www.math.uni.wroc.pl/\textasciitilde gogolok/}
\author[P. KOWALSKI]{Piotr Kowalski$^{\spadesuit}$}
\thanks{$^{\spadesuit}$ Supported by the Narodowe Centrum Nauki grant no. 2018/31/B/ST1/00357 and by the T\"{u}bitak 1001 grant no. 119F397.}
\address{$^{\spadesuit}$Instytut Matematyczny\\
Uniwersytet Wroc{\l}awski\\
pl. Grunwaldzki 2, 50-384
\\
Wroc{\l}aw\\
Poland}
\email{pkowa@math.uni.wroc.pl} \urladdr{http://www.math.uni.wroc.pl/\textasciitilde pkowa/ }

\DeclareMathOperator{\locus}{locus}
\DeclareMathOperator{\acl}{acl} \DeclareMathOperator{\dcl}{dcl} \DeclareMathOperator{\scf}{SCF}
 \DeclareMathOperator{\aut}{Aut} \DeclareMathOperator{\id}{id}
 \DeclareMathOperator{\fr}{Fr}

 \DeclareMathOperator{\alg}{alg}

\DeclareMathOperator{\tp}{tp}

\DeclareMathOperator{\spec}{Spec}\DeclareMathOperator{\rat}{rat}

\DeclareMathOperator{\Alg}{Alg}\DeclareMathOperator{\Hom}{Hom}

\DeclareMathOperator{\cf}{CF}

\newtheorem{theorem}{Theorem}[section]
\newtheorem{prop}[theorem]{Proposition}
\newtheorem{lemma}[theorem]{Lemma}
\newtheorem{cor}[theorem]{Corollary}

\theoremstyle{definition}
\newtheorem{definition}[theorem]{Definition}
\newtheorem{example}[theorem]{Example}
\newtheorem{remark}[theorem]{Remark}
\newtheorem{question}[theorem]{Question}
\newtheorem{notation}[theorem]{Notation}
\newtheorem{assumption}[theorem]{Assumption}

\begin{document}
\newcommand{\lili}{\underleftarrow{\lim }}
\newcommand{\coco}{\underrightarrow{\lim }}
\newcommand{\twoc}[3]{ {#1} \choose {{#2}|{#3}}}
\newcommand{\thrc}[4]{ {#1} \choose {{#2}|{#3}|{#4}}}
\newcommand{\Zz}{{\mathds{Z}}}
\newcommand{\Ff}{{\mathds{F}}}
\newcommand{\Cc}{{\mathds{C}}}
\newcommand{\Rr}{{\mathds{R}}}
\newcommand{\Nn}{{\mathds{N}}}
\newcommand{\Qq}{{\mathds{Q}}}
\newcommand{\Kk}{{\mathds{K}}}
\newcommand{\Pp}{{\mathds{P}}}
\newcommand{\ddd}{\mathrm{d}}
\newcommand{\Aa}{\mathds{A}}
\newcommand{\dlog}{\mathrm{ld}}
\newcommand{\ga}{\mathbb{G}_{\rm{a}}}
\newcommand{\gm}{\mathbb{G}_{\rm{m}}}
\newcommand{\gaf}{\widehat{\mathbb{G}}_{\rm{a}}}
\newcommand{\gmf}{\widehat{\mathbb{G}}_{\rm{m}}}
\newcommand{\ka}{{\bf k}}
\newcommand{\ot}{\otimes}
\newcommand{\si}{\mbox{$\sigma$}}
\newcommand{\ks}{\mbox{$({\bf k},\sigma)$}}
\newcommand{\kg}{\mbox{${\bf k}[G]$}}
\newcommand{\ksg}{\mbox{$({\bf k}[G],\sigma)$}}
\newcommand{\ksgs}{\mbox{${\bf k}[G,\sigma_G]$}}
\newcommand{\cks}{\mbox{$\mathrm{Mod}_{({A},\sigma_A)}$}}
\newcommand{\ckg}{\mbox{$\mathrm{Mod}_{{\bf k}[G]}$}}
\newcommand{\cksg}{\mbox{$\mathrm{Mod}_{({A}[G],\sigma_A)}$}}
\newcommand{\cksgs}{\mbox{$\mathrm{Mod}_{({A}[G],\sigma_G)}$}}
\newcommand{\crats}{\mbox{$\mathrm{Mod}^{\rat}_{(\mathbf{G},\sigma_{\mathbf{G}})}$}}
\newcommand{\crat}{\mbox{$\mathrm{Mod}^{\rat}_{\mathbf{G}}$}}
\newcommand{\cratinv}{\mbox{$\mathrm{Mod}^{\rat}_{\mathbb{G}}$}}
\newcommand{\ra}{\longrightarrow}
\newcommand{\bdcf}{\mathcal{B}-\cf}

\def\Ind#1#2{#1\setbox0=\hbox{$#1x$}\kern\wd0\hbox to 0pt{\hss$#1\mid$\hss}
\lower.9\ht0\hbox to 0pt{\hss$#1\smile$\hss}\kern\wd0}

\def\ind{\mathop{\mathpalette\Ind{}}}

\begin{abstract}
We introduce the notion of a \emph{coordinate $\ka$-algebra scheme} and the corresponding notion of a \emph{$\mathcal{B}$-operator}. This class of operators includes endomorphisms and derivations of the Frobenius map, and it generalizes the operators related to $\mathcal{D}$-rings from \cite{MS2} as well. We classify the (coordinate) $\ka$-algebra schemes for a perfect field $\ka$ and we also discuss the model-theoretic properties of fields with $\mathcal{B}$-operators.
\end{abstract}

\maketitle

\section{Introduction}
In this paper, we study \emph{operators} of a certain kind on rings and fields. There are many versions of this notion in the literature, we list several of them below.
\begin{enumerate}
\item Bia{\l}ynicki-Birula's notion of \emph{fields with operators}, which are fields together with a family of automorphisms and derivations \cite{BBop}.

\item Buium's notion of \emph{jet operators}, which are functions on rings satisfying additive and multiplicative laws given by polynomials \cite{Bui3}.

\item The notion of \emph{$\mathcal{D}$-rings} introduced by Moosa and Scanlon \cite{MS2} and the equivalent notion considered by Kamensky in \cite{Kam}.

\item Hardouin's notion of \emph{iterative $q$-difference operators} see \cite{Hard}.

\end{enumerate}
The situations considered in Items $(2)$ and $(3)$ above jointly generalize the set-up from Item $(1)$ (actually, by \cite{K4}, the ``intersection'' of $(2)$ and $(3)$ is exactly $(1)$).
We introduce here a certain class of operators, which was invented in an attempt of bringing into a common framework the set-ups from Items $(2)$ and $(3)$ above. A more concrete motivation was to generalize the definition from \cite{MS2} to a case including derivations of Frobenius and to find out whether the model-theoretic results from \cite{K2}, \cite{MS2}, and \cite{BHKK} still hold in this wider context (such a possibility was alluded to in \cite[Section 5.3]{BHKK}). We explain it in a more detailed way below.

We start from the following example.
If $\partial$ is a derivation on a field $K$ of positive characteristic $p$, then $\partial':=\fr_K\circ \partial$ is a \emph{derivation of Frobenius}, i.e. $\partial'$ is additive and satisfies the following ``twisted Leibniz rule'':
$$\partial'(xy)=x^p\partial'(y)+y^p\partial'(y).$$
As it was discussed by the second author in \cite{K2}, derivations of Frobenius are usually \emph{not} of the form $\fr_K\circ \partial$ and they have a reasonable model theory. However, such operators do not fit to the set-up from \cite{MS2} (it is briefly explained in \cite[Section 5.3]{BHKK}), but they are examples of jet operators from \cite{Bui3}. More generally (we use here the definitions from \cite{BHKK}, which are equivalent to the ones from \cite{MS2}), if we replace ``derivations'' with ``$B$-operators'', where $B$ is an arbitrary finite $\Ff_p$-algebra, then an appropriate composition of a $B$-operator and the Frobenius map should be considered as a ``twisted version'' of a $B$-operator. Such operators neither fit to the set-up from \cite{MS2} nor they are the jet operators from \cite{Bui3}.

The aim of this paper is to introduce a natural new class of operators (we call them \emph{$\mathcal{B}$-operators}), which satisfies the following properties:
\begin{itemize}
  \item it includes the operators coming from $\mathcal{D}$-rings considered in \cite{MS2};

  \item it includes jet operators  from \cite{Bui3} (at least on fields);

  \item it covers derivations of Frobenius and (more generally) twisted $B$-operators mentioned above;

  \item it allows a model-theoretic analysis similar as in \cite{K2}, \cite{MS2}, and \cite{BHKK}.
\end{itemize}
It is a good moment to remark that in the case of characteristic $0$, our set-up of $\mathcal{B}$-operators coincides with the set-up from \cite{MS2} (see Corollary \ref{charac0}). Therefore, all the new interesting phenomena related to $\mathcal{B}$-operators take place in the case of positive characteristic.

This paper has two aspects: the algebraic one and the model-theoretic one, which is exemplified by the following two main results.
\begin{enumerate}
\item[(i)] A classification of $\ka$-algebra schemes for a perfect field $\ka$ (Theorem \ref{classop}) extending the two-dimensional classification from \cite{K4}. This result yields a comparison (given by a sequence of Frobenius maps) between arbitrary $\ka$-algebra schemes and the $\ka$-algebra schemes related to $\mathcal{D}$-rings from \cite{MS2}.

\item[(ii)] A criterion for $\mathcal{B}$-operators (following the one from \cite{BHKK}) implying the existence of a model companion of the theory of fields with $\mathcal{B}$-operators  (Theorem \ref{supmain}) as well as its model-theoretic description (Section \ref{secstable}).
\end{enumerate}
We also extend some of the technical framework (the theory of prolongations) considered in \cite{MS2} into our new class of operators,
which requires generalizing the classical notion of the Weil restriction (in the affine case).

The paper is organized as follows. In Section \ref{secrings}, we collect the necessary results about $\ka$-algebra schemes and we place the ring schemes used in \cite{MS2} in our more general context as well. In Section \ref{secoper}, we describe our setting for $\mathcal{B}$-operators using the class of ring schemes, which was specified in Section \ref{secrings}. In Section \ref{secprol}, we show that the prolongations still exist in this new set-up and that they have good geometric properties. In Section \ref{secmt}, we analyze the  model theory of fields with $\mathcal{B}$-operators. To show the existence results, we use the prolongations from Section \ref{secprol}. We discuss some issues related to the negative results in Section \ref{secneg} and we finish with the quantifier elimination and stability results  (Section \ref{secstable}).

We would like to thank the referee for a careful reading of our paper and
many useful suggestions.

\section{Ring schemes}\label{secrings}
All rings considered in this paper are commutative and with $1$, except the ring of skew-polynomials which is always explicitly mentioned. Let $R$ be a ring, $I$ be an ideal in $R$, and $n$ be a positive integer. By $I^n$, we denote the $n$-fold product of $I$ with itself (product of ideals). If $R$ is of prime characteristic $p$, then $\fr_R$ denotes the Frobenius endomorphism:
$$\fr_R:R\ra R,\ \ \ \ \ \ \fr_R(r)=r^p,$$
and $\fr_R(I)$ is the image of the ideal $I$ by this endomorphism. However, instead of $\fr_R(R)$, we often write $R^p$. For any set $X$, $X^{\times n}$ denotes the $n$-th Cartesian power of $X$ (similarly for products of schemes, however, we still write ``$\Aa_{\ka}^n,\ga^n$'' rather than  ``$\Aa_{\ka}^{\times n},\ga^{\times n}$'').
We fix a base ring $\ka$, which will usually be a field (often of positive characteristic).

In this section, we define certain functors ``governing'' the class of operators, which will be introduced in the next section.
We also prove a classification result (Theorem \ref{classop}) extending the main theorem of \cite{K4}.

\subsection{Our categorical set-up}
We will consider $\ka$-algebra schemes with some extra data. The notion of a $\ka$-algebra scheme is natural, but it does not seem to
be very well established. For possible references, we could only find \cite[page 148]{ShGal}
and \cite[Section 3]{MS2} (it is called an ``$\mathbb{S}$-algebra scheme'' in \cite{MS2}).
Therefore, we will recall this notion below.

We start from the notion of an affine \emph{ring scheme} over $\ka$, which is a representable functor from the category of $\ka$-algebras, denoted $\mathrm{Alg}_{\ka}$, to the category of rings, see e.g. \cite[Lecture 26]{mumlectcur}. Since most of the time we assume that $\ka$ is a field and we consider ring schemes of finite type over $\ka$ (as schemes), we could also use the term ``algebraic ring'' as in \cite{algring}.
\begin{definition}\label{defs}
\begin{enumerate}
\item We denote the identity functor on the category $\mathrm{Alg}_{\ka}$ by:
$$\mathbb{S}_{\ka}=\id:\mathrm{Alg}_{\ka}\ra \mathrm{Alg}_{\ka},\ \ \ \ \mathbb{S}_{\ka}(R)=R.$$
Since $\mathbb{S}_{\ka}$ is represented by the affine line $\Aa^1_{\ka}$, it is a ring scheme over $\ka$.

\item By a \emph{$\ka$-algebra scheme}, we mean a morphism
$$\iota:\mathbb{S}_{\ka}\ra \mathcal{B}$$
of ring schemes over $\ka$. As in the case of $\ka$-algebras, we often say ``$\mathcal{B}$ is a $\ka$-algebra scheme'' or  ``$(\mathcal{B},\iota)$ is a $\ka$-algebra scheme''.
\end{enumerate}
\end{definition}
\begin{remark}\label{ralg}
Let $(\mathcal{B},\iota)$ be a $\ka$-algebra scheme and $R$ be a $\ka$-algebra.
\begin{enumerate}
\item The structure map evaluated on $R$:
$$\iota_R:\mathbb{S}_{\ka}(R)=R\ra \mathcal{B}(R)$$
gives $\mathcal{B}(R)$ the structure of an $R$-algebra.

\item We denote the extension of scalars from $\ka$ to $R$ in the following way:
$$\mathcal{B}_R:=\mathcal{B}\times_{\spec(\ka)}\spec(R).$$
Clearly, $\mathcal{B}_R$ has a natural structure of an $R$-algebra scheme.

\item There are ring schemes over a field $\ka$, which do not have a $\ka$-algebra scheme structure. The main example is the ring scheme of \emph{n-truncated  Witt vectors} $W_n$ for $n>1$ (see e.g. \cite[Lecture 26, Section 2]{mumlectcur}).

\end{enumerate}
\end{remark}
\begin{example}\label{example1}
Any finite free $\ka$-algebra $B$ yields a $\ka$-algebra scheme, which we denote by $B_{\otimes}$, in the following way:
$$B_{\otimes}:\mathrm{Alg}_{\ka}\ra \mathrm{Alg}_{\ka},\ \ \ \ B_{\otimes}(R)=R\otimes_{\ka}B$$
(see \cite[Remark 2.3]{MS}). In particular, we have $\ka_{\otimes}=\mathbb{S}_{\ka}$.
\end{example}
We state below an important result about the additive group scheme of a $\ka$-algebra scheme. If $\ka$ is an algebraically closed field, then this result appears in \cite{algring} and it can be easily transferred to the case of a  perfect field $\ka$, using for example the theory described in \cite{sprinlin}. However, in the case of an arbitrary base field $\ka$, we could not find such a result in the literature.
\begin{prop}\label{additive}
We assume that $\ka$ is a field and $\mathcal{B}$ is a $\ka$-algebra scheme, which is of finite type and connected (as a scheme over $\ka$). Then, for some positive integer $e$, we have the following isomorphism of group schemes over $\ka$:
$$(\mathcal{B},+)\cong \ga^e.$$
\end{prop}
\begin{proof}
We consider two cases.
\smallskip
\\
\textbf{Case 1} $\mathrm{char}(\ka)=0$.
\\
In this case, we do not use the $\ka$-algebra scheme assumption, that is assuming that $\mathcal{B}$ is a connected ring scheme of finite type is enough.
By \cite[Corollary 4.5]{algring}, we get that for some positive integer $e$ (see our notation from Remark \ref{ralg}(2)):
$$\left(\mathcal{B}_{\ka^{\alg}},+\right)\cong \ga^e.$$
Therefore, $(\mathcal{B},+)$ is an elementary unipotent group (see \cite[Section 3.4]{sprinlin}). Since $\ka$ is perfect, the elementary unipotent group $(\mathcal{B},+)$ is $\ka$-split (see \cite[Corollary 14.3.10]{sprinlin}). Finally, by \cite[13.34]{milnered}, we obtain that $(\mathcal{B},+)\cong \ga^e$.
\smallskip
\\
\textbf{Case 2}  $\mathrm{char}(\ka)=p>0$.
\\
For any $\ka$-algebra $R$, the group $(\mathcal{B}(R),+)$ has exponent $p$, since (by Remark \ref{ralg}) it has a structure of a vector space over $\ka$. Therefore, $(\mathcal{B},+)$ is an elementary unipotent group again. We use now the $\ka$-algebra scheme structure and obtain that there is an
action of $\gm$ on $(\mathcal{B},+)$ by group scheme automorphisms such that $0$ is the only rational point
 of $(\mathcal{B},+)$ fixed by $\gm$. Since $(\mathcal{B},+)$ is connected, we are exactly in  the situation of the assumptions of
\cite[Corollary 14.4.2]{sprinlin} and we get that $(\mathcal{B},+)$ is $\ka$-split. Therefore, we can conclude the proof as in Case 1 above.
\end{proof}
From now on, for any $\ka$-algebra scheme $\mathcal{B}$ satisfying the assumptions from Proposition \ref{additive}, we just assume that $(\mathcal{B},+)=\ga^e$ (for some positive integer $e$).

Our main definition is below.
\begin{definition}\label{bdef}
A \emph{coordinate $\ka$-algebra scheme} is a $\ka$-algebra scheme $(\mathcal{B},\iota)$ such that:
\begin{enumerate}
\item there is a positive integer $e$ and a fixed isomorphism of group schemes over $\ka$ such that:
$$(\mathcal{B},+)\cong \ga^e,$$
\item there is a fixed morphism of $\ka$-algebra schemes:
$$\pi:\mathcal{B}\ra \mathbb{S}_{\ka}$$
(that is: a morphism of ring schemes satisfying $\pi\circ\iota=\id$) such that under the identification $(\mathcal{B},+)=\ga^e$, $\pi$ is the projection on the first coordinate.
\end{enumerate}
\end{definition}
\begin{remark}
\begin{enumerate}
\item If $\ka$ is a field, then, using Proposition \ref{additive}, we can replace Item $(1)$ in Definition \ref{bdef} with the statement: ``$\mathcal{B}$ is a connected scheme of finite type over $\ka$''.

\item If we make an extra assumption that for each $\ka$-algebra $R$, the $R$-module structure (from Remark \ref{ralg}(1)) on $\mathcal{B}(R)=R^{\times e}$ coincides with the product $R$-module structure (in other words, the isomorphism from Proposition \ref{additive} is an isomorphism of ``$\ka$-module schemes''), then we get exactly ``the basic data'' from \cite[page 5]{MS2}. Therefore, preservation of the scalar multiplication by the isomorphism from Proposition \ref{additive} gives the main dividing line between the approach from \cite{MS2} and the approach here.

\item In the situation from \cite{MS2}, the entire ring scheme data is given by the finite $\ka$-algebra $B:=\mathcal{B}(\ka)$ and the use of ring schemes can be avoided (see again \cite[page 5]{MS2}). It is not the case here, since there is no way to encode the functor $\mathcal{B}$ from Definition \ref{bdef} using only the $\ka$-algebra $\mathcal{B}(\ka)$ (see Example \ref{firstex}(2)). However, if $\ka$ is a perfect field, then the functor $\mathcal{B}$ \emph{is} given by $\mathcal{B}(\ka)$ and a certain sequence of powers of Frobenius maps (see Theorem \ref{classop}).

\item It is a good moment to compare Kamensky's set-up from \cite{Kam} with the ring scheme functor discussed in Items $(1)$ and $(2)$ above. The above comments on a finite $\ka$-algebra $B$ controlling the corresponding functor apply to \cite{Kam} as well. The main difference between \cite{MS2} and \cite{Kam} lies in the choice of the affine scheme over $\ka$ corresponding to the $\ka$-algebra $B$: the finite $\ka$-scheme $\spec(B)$ is considered in \cite{Kam}, in particular, there are no ring schemes in \cite{Kam}.
\end{enumerate}
\end{remark}
\begin{example}\label{firstex}
\begin{enumerate}
\item Assume that $B$ is as in Example \ref{example1} and that we have a $\ka$-algebra map $\pi_B:B\to \ka$. Then, the functor $B_{\otimes}$ from Example \ref{example1} is a coordinate $\ka$-algebra scheme and this is exactly the type of ring schemes, which is considered in \cite{MS2}. In Proposition \ref{eqcon}, we give several equivalent conditions characterizing this type of (coordinate) $\ka$-algebra schemes.

\item We consider now our motivating example, which originates from \cite{K2}. We assume that $\ka=\Ff_p$ and define $\mathcal{B}=\Aa^2_{\ka}$ (as a $\ka$-scheme) with the following ring scheme structure:
$$\ \ \ \ \ \ \ \ \ \ (x,x')+(y,y'):=(x+y,x'+y'),\ \ \ (x,x')\cdot (y,y'):=\left(xy,x^py'+y^px'\right).$$
This ring scheme has the following $\ka$-algebra scheme structure:
$$\iota:\mathbb{S}_{\ka}\ra \mathcal{B},\ \ \ \ \iota(x)=(x,0);$$
and $\pi$ is the projection on the first coordinate making $\mathcal{B}$ a coordinate $\ka$-algebra scheme.

It is clear that $\mathcal{B}$ is not of the form $B_{\otimes}$ for (recall that $\ka=\Ff_p$ here):
$$B:=\mathcal{B}(\ka)\cong \ka[X]/(X^2),$$
since if $R$ is a $\ka$-algebra, then $\mathcal{B}(R)$ need not be isomorphic to $R[X]/(X^2)$ as an $R$-algebra. For example, if $R=K$ is a field of is infinite imperfection degree, then $\dim_K\mathcal{B}(K)$ is infinite as well.

\item Similarly as in Item $(2)$ above, we get a ring scheme structure $\mathcal{B}$ on $\Aa^2_{\ka}$ associated to any jet operator on $\ka$ (see \cite{Bui3} and \cite[Definition 1.1]{K4}). In this case, there is still a morphism $\pi:\mathcal{B}\to \mathbb{S}_{\ka}$, but $\mathcal{B}$ need \emph{not} be a $\ka$-algebra scheme (see \cite[Example(d)]{Bui3}).

    However, if $\ka$ is a field, then by a classification result of the second author \cite[Theorem 2.1]{K4}, we get that $\mathcal{B}$ has a $\ka$-algebra scheme structure $\iota$ compatible with $\pi$, so $\mathcal{B}$ becomes a coordinate $\ka$-algebra scheme.
\end{enumerate}
\end{example}

\subsection{Twist and transport}
Suppose now that we have two ring homomorphisms $\pi:S\to T,\iota:T\to S$ such that $\pi\circ \iota=\id$ and we also have a ring endomorphism $F:T\to T$. Then, we can ``twist'' the ring multiplication on $S$ using $F$, as it is explained below.
\begin{lemma}\label{twist}
We define first the following three maps from $S$ to $S$ ($x\in S$):
$$x^T:=\iota(\pi(x)),\ \ \ \ \ x^{\perp}:=x-x^T,\ \ \ \ \ x^F:=\iota(F(\pi(x))).$$
A new multiplication on $S$ is defined in the following way ($x,y\in S$):
$$x*y := x^Ty^T+x^Fy^{\perp}+x^{\perp}y^F+x^{\perp}y^{\perp}.$$
Then, $(S,+,*)$ is a ring (commutative and with $1$) and the maps $\pi,\iota$ are still ring homomorphisms with respect to the ring $(S,+,*)$.
\end{lemma}
\begin{proof}
The operation $*$ is clearly commutative. The three maps defined above are additive, hence the distributivity follows. Since $1^T=1=1^F$ and $1^{\perp}=0$, we get that $1$ is a neutral element of $*$ as well. Regarding the associativity of $*$, it can be easily computed that:
\begin{IEEEeqnarray*}{rCl}
(x*y)*z & = & x^Ty^Tz^T+x^Fy^Fz^{\perp}+x^Fy^{\perp}z^F+x^{\perp}y^Fz^F \\
 &+& x^{\perp}y^{\perp}z^F+x^{\perp}y^Fz^{\perp}+x^Fy^{\perp}z^{\perp}+x^{\perp}y^{\perp}z^{\perp}\\
 &= & x*(y*z).
\end{IEEEeqnarray*}
Since the map $\perp$ vanishes on the image of $\iota$, we get $\pi$ and $\iota$ are multiplicative with respect to $*$ as well.
\end{proof}
\begin{remark}\label{transport}
In the set-up above, if we additionally have a bijection:
$$\widetilde{F}:\ker(\pi)\ra \ker(\pi),$$
which is additive, multiplicative, and for $a\in T,b\in \ker(\pi)$ it satisfies:
$$\widetilde{F}(\iota(a)b)=\iota(F(a))\widetilde{F}(b),$$
then the map:
$$x\mapsto x^T+\widetilde{F}\left(x^{\perp}\right)$$
is an isomorphism between the original ring $S$ and the new ring, which was obtained using Lemma \ref{twist}. The isomorphism above commutes with the maps $\iota$ and $\pi$.
\end{remark}
The assumptions from Remark \ref{transport} look unlikely to be satisfied, but they are satisfied in our coordinate $\ka$-algebra scheme situation. If $\ka$ is a field of positive characteristic, then any coordinate $\ka$-algebra scheme $\mathcal{B}$ can be ``twisted'' (Example \ref{firstex}(2) is the motivating case) as in Lemma \ref{twist} with respect to the Frobenius morphism on $\mathbb{S}_{\ka}$, which we denote by $\fr_{\mathbb{S}_{\ka}}$.
\begin{definition}\label{twistdef}
Assume that $\ka$ is a field of positive characteristic, $\mathcal{B}$ is a coordinate $\ka$-algebra scheme, and $n$ is a positive integer. Then,
we denote by $\mathcal{B}^{\fr^n}$ the coordinate $\ka$-algebra scheme obtained from $\mathcal{B}$ using twisting (as in Lemma \ref{twist}) by the morphism $\fr_{\mathbb{S}_{\ka}}^n$.
\end{definition}
\begin{remark}\label{transportfr}
In the case of Definition \ref{twistdef}, we can get an easier proof of Lemma \ref{twist} using Remark \ref{transport}. We assume for simplicity that $\iota(x)=(x,0,\ldots,0)$. If $\ka\subseteq K$ is a field extension such that $K$ is perfect, then we define:
$$\widetilde{F}:=\left(0,\fr_K,\ldots,\fr_K\right):\ker(\pi)(K)\ra \ker(\pi)(K).$$
Since $\fr_K$ is an automorphism of the field $K$, it commutes with all the maps from Lemma \ref{twist}, hence it satisfies the conditions for $\widetilde{F}$ from Remark \ref{transport}.
\end{remark}
\begin{example}\label{trexam}
We describe the twist in two most basic cases.
\begin{enumerate}
\item Suppose that $\mathrm{char}(\ka)=p>0$ and
$$\mathcal{B}=\left(\ka[X]/(X^2)\right)_{\otimes}.$$
It is easy to see that the multiplication in $\mathcal{B}^{\fr^n}$ is given by:
$$\left(x,x'\right)*\left(y,y'\right)=\left(xy,x^{p^n}y'+x'y^{p^n}\right)$$
and we get the coordinate $\ka$-algebra scheme discussed in Example \ref{firstex}(2).

It was noted in the last paragraph of \cite{K4}, that if $n\neq m$, then $\mathcal{B}^{\fr^n}$ and $\mathcal{B}^{\fr^m}$ are not isomorphic as ``ring schemes over $\pi$'' (that is, there is no ring scheme isomorphism commuting with the fixed projection maps as in Definition \ref{bdef}(2)), and it is not difficult to check that they are not isomorphic as ring schemes either.

\item Let us consider the case of
$$\mathcal{B}=\mathbb{S}_{\ka}\times \mathbb{S}_{\ka}=(\ka\times \ka)_{\otimes}$$
and we fix $n>0$. Then, it is easy to check that the morphism:
$$\left(x,x'\right)\mapsto \left(x,x-x^{p^n}+x'\right)$$
is an isomorphism between $\mathcal{B}$ and $\mathcal{B}^{\fr^n}$ in the category of ``ring schemes over $\pi$'' (see the comment in Item $(1)$ above). It is clear that  $\mathcal{B}$ and $\mathcal{B}^{\fr^n}$ are not isomorphic as $\ka$-algebra schemes, since they are not even isomorphic as ``$\ka$-module schemes'' (it is enough to look at the dimensions over a non-perfect field $K$ extending $\ka$).
\end{enumerate}
\end{example}
Using Remark \ref{transportfr}, we can try to generalize the construction from Definition \ref{twistdef} and attempt to transport the multiplicative structure on $\mathcal{B}$ by an additive endomorphism of the form:
$$\Psi=\left(\id,\fr^{n_2},\ldots,\fr^{n_e}\right):\ga^e\ra \ga^e$$
for some $n_2,\ldots,n_e\in \Nn$. If this transported multiplication is given by polynomials (that is: if the inverse of the Frobenius map does not appear), then we get a new coordinate $\ka$-algebra scheme.
The examples below show that this new multiplication need not be given by polynomials. By our classification result (Theorem \ref{classop}), \emph{all} coordinate $\ka$-algebra schemes come from such a transport of a coordinate $\ka$-algebra schemes of the form $B_{\otimes}$ as in Example \ref{firstex}(1).
\begin{example}\label{twistex}
We analyze here two examples of transport.
\begin{enumerate}
\item If $B=\ka[X]/(X^3)$ and $\Psi=(\id,\fr^2,\fr)$, then it is easy to see that the transported multiplication is \emph{not} given by polynomials. More generally, if $\Psi=\left(\id,\fr^{m},\fr^{n}\right)$, then $\Psi$ yields the transported multiplicative structure given by a polynomials if and only if $m\leqslant n$. In this case, we get the following formula for the transported multiplication:
    $$\ \ \ \ \ \ \ \ \ \ \ \ \ \left(x_1,x_2,x_3\right)\ast \left(y_1,y_2,y_3\right)=\left(x_1y_1,x_1^{p^m}y_2+y_1^{p^m}x_2,x_1^{p^n}y_3+\left(x_2y_2\right)^{p^{n-m}}+y_1^{p^n}x_3\right).$$
    For $m=1$ and $n=2$, it seems to be the simplest example of a ``new'' coordinate $\ka$-algebra scheme, which does not come from a Frobenius twist (as in Definition \ref{twistdef}) of a coordinate $\ka$-algebra scheme of the form $B_{\otimes}$.

\item Let $\mathcal{B}$ be a ``difference $\ka$-algebra scheme'', that is $\mathcal{B}\cong \mathbb{S}_{\ka}^{\times e}$, but the $\ka$-algebra scheme structure is given by $\iota(x)=(x,0,\ldots,0)$, so the multiplication formula is as follows:
    $$\left(x_1,x_2,\ldots,x_e\right)\cdot\left(y_1,y_2,\ldots,y_e\right)$$
    $$=\left(x_1y_1,x_1y_2+y_1x_2+x_2y_2,\ldots,y_1x_e+x_1y_e+x_ey_e\right).$$
    For any sequence of non-negative integers $n_2,\ldots,n_e$, the transport by
    $$\Psi:=\left(\id,\fr^{n_2},\ldots,\fr^{n_e}\right)$$
    is given by polynomials:
    $$\left(x_1,x_2,\ldots,x_e\right)\ast\left(y_1,y_2,\ldots,y_e\right)$$
      $$=\left(x_1y_1,y_1^{p^{n_2}}x_2+x_1^{p^{n_2}}y_2+x_2y_2,\ldots,
      y_1^{p^{n_e}}x_e+x_1^{p^{n_e}}y_e+x_ey_e\right).$$
   The transported $\ka$-algebra scheme is isomorphic to the following fiber product of ring schemes over $\ka$:
   $$\left(\mathbb{S}_{\ka}^{\times 2}\right)^{\fr^{n_2}}\times_{\ka} \ldots \times_{\ka} \left(\mathbb{S}_{\ka}^{\times 2}\right)^{\fr^{n_e}}.$$
   By Example \ref{trexam}(2), the above $\ka$-algebra scheme is isomorphic as a ``ring scheme over $\pi$'' to $\mathbb{S}_{\ka}^{\times e}$, which we will use in the proof of Theorem \ref{supmain}.
\end{enumerate}
\end{example}

\subsection{Classification of $\ka$-algebra schemes}
In this subsection, we show that if $\ka$ is a perfect field, then $\ka$-algebra schemes can be described by a finite $\ka$-algebra (as in Example \ref{example1}) and a finite sequence of non-negative integers corresponding to the transport by a sequence of powers of the Frobenius maps as in Example \ref{twistex}.

For a field $\ka$, by $\ka[\fr]$ we mean the ring of ``Frobenius polynomials'' with the usual addition and the multiplication given by the composition of polynomials. More precisely, if $\mathrm{char}(\ka)=0$, then we set $\ka[\fr]$ as $\ka$. If $\mathrm{char}(\ka)=p>0$, then $\ka[\fr]$ consists of polynomials with coefficients in $\ka$, which are of the following form:
$$a_0X+a_1X^p+\ldots +a_nX^{p^n}.$$
It is well-known that $\ka[\fr]$ is isomorphic to the ring of endomorphisms of the additive group scheme $\ga$ over $\ka$ (additive polynomials). If $\ka$ is a positive characteristic field which is not a prime field, then the ring $\ka[\fr]$ is not commutative. This ring coincides with the skew-polynomial ring with respect to the ring endomorphism $\fr_{\ka}:\ka\to \ka$ (the Frobenius map on $\ka$).

We fix a positive integer $e$.
\begin{lemma}\label{diagonal}
We assume that $\ka$ is a perfect field and $\phi \colon \ga^{e} \to \ga^{e}$ is a morphism of group schemes. Then, there are $\alpha, \beta \in \aut(\ga^e)$ such that the composition morphism $\alpha\phi\beta \colon \ga^{e} \to \ga^{e}$ is given by
$$\left( x_1 , \dotsc , x_e \right) \mapsto \left( s_1\left( x_1  \right), \dotsc , s_e\left( x_e  \right) \right)$$
for some monic additive polynomials $s_1, \dotsc , s_e \in \ka[\fr]$.
\end{lemma}
\begin{proof}
The morphism $\phi$ corresponds to an $e\times e$ matrix $M$ with coefficients in $\ka[\fr]$.
 There is an algorithm (similar to Gaussian  elimination) putting $M$ into the diagonal form, which yields the desired representation of $\phi$. We present the details for $e=2$, as the general case is analogous.

Let
$$M=\begin{bmatrix}
s_{11} & s_{12} \\
s_{21} & s_{22}
\end{bmatrix}$$
for some $s_{ij}\in \ka[\fr]$. Firstly, we will get rid of the $\left( 2, 1 \right)$-entry.

Assume that the degree of $s_{21}$ is not lower than the degree of $s_{11}$; if this is not the case, we swap the rows of $M$ using multiplication  on the left by the matrix $\begin{bmatrix}
0 & 1 \\
1 & 0
\end{bmatrix}$.  By \cite[Lemma 3.3.2(i)]{sprinlin} (originating from \cite{ore}), there are some $p, r\in \ka[\fr]$ such that $s_{21}=ps_{11}+r$ with $\deg r < \deg s_{21}$. By setting $\alpha = \begin{bmatrix}
1 & 0\\
-p & 1
\end{bmatrix}$,
we get that $\alpha \phi = \begin{bmatrix}
s_{11} & s_{12} \\
r & -ps_{12}+s_{22}
\end{bmatrix}$. We reduced the degree of the $\left( 2, 1 \right)$-entry and by iterating this process we annihilate this entry.

By \cite[Lemma 3.3.2(ii)]{sprinlin}, the left division with remainder is possible in $\ka[\fr]$ (perfectness of $\ka$ matters here) and we get rid of the $\left( 1, 2 \right)$-entry  in a similar way.

Therefore, we can assume that $\phi$ corresponds to a diagonal matrix. By composing with a morphism of the form
$$\left(x_1,\ldots,x_e\right)\mapsto \left(x_1/a_1,\ldots,x_e/a_e\right),$$
we obtain that the additive polynomials from the diagonal of $M$ are monic.
\end{proof}

\begin{lemma}\label{diagonaltofr}
Let
$$\phi=(s_1,\ldots,s_e) \colon \mathcal{A} \ra \mathcal{B}$$
be a morphism of ring schemes over a field $\ka$ such that the additive group schemes of $\mathcal{A}$ and $\mathcal{B}$ are isomorphic to $\ga^{e}$, and $s_1,\ldots,s_e$ are as in Lemma \ref{diagonal}. Then, for each $i$ we have that $s_i$ is a power of the Frobenius morphism or $s_i=0$.
\end{lemma}
\begin{proof}
For any (additive) morphism $s:\ga\to \ga$ of additive group schemes over $\ka$, $s$ is a power of the Frobenius morphism or $s=0$ if and only if the same holds for the induced (base-change) morphism of additive groups schemes over the algebraic closure $\ka^{\alg}$:
$$s_{\ka^{\alg}}:\mathbb{G}_{\mathrm{a},\ka^{\alg}}\ra \mathbb{G}_{\mathrm{a},\ka^{\alg}}.$$
Therefore, we can assume that the field $\ka$ is algebraically closed.
By \cite[Proposition 2.1]{algring}, the ideal $\ker(\phi)$ is a connected algebraic variety. If there is $i$ such that $s_i$ is neither a power of Frobenius nor $s_i=0$, then $s_i$ is not a monomial and $\ker(s_i)$ is a finite non-trivial group. Since we have:
$$\ker(\phi)=\ker(s_1)\times \ldots \times \ker(s_e),$$
we get that $\ker(\phi)$ is not connected, which is a contradiction.
\end{proof}
\begin{theorem}\label{dime}
Assume that the field $\ka$ is perfect and $\mathcal{B}$ is a $\ka$-algebra scheme such that $(\mathcal{B},+)\cong \ga^e$. Then, we have:
$$\dim_{\ka}(\mathcal{B}(\ka))=e.$$
\end{theorem}
\begin{proof}\renewcommand{\qedsymbol}{}
For any $b\in \mathcal{B}(\ka)$ considered as a scheme morphism $\spec(\ka)\to \mathcal{B}$, we denote by $s_b$ the following composition morphism of group schemes ($\mathcal{B}$ is considered below with the additive group scheme structure):
\begin{equation*}
\xymatrix{\ga \ar[r]^-{\cong} &  \ga\times_{\spec(\ka)}\spec(\ka) \ar[r]^-{\iota\times b} & \mathcal{B}\times_{\spec(\ka)}\mathcal{B} \ar[r]^-{m}  & \mathcal{B},}
\end{equation*}
where $m$ is the ring scheme multiplication in $\mathcal{B}$. For any finite tuple $\bar{b}=(b_1,\ldots,b_n)$ of elements of $\mathcal{B}(\ka)$, we denote by $s_{\bar{b}}$ the following morphism of group schemes:
\begin{equation*}
\xymatrix{\ga^n \ar[rr]^-{s_{b_1}\times \ldots \times s_{b_n}} & & \mathcal{B}^{\times n} \ar[rr]^-{+}  & & \mathcal{B},}
\end{equation*}
where $+$ is the ring scheme addition in $\mathcal{B}$.
\end{proof}
\noindent
{\bf Claim 1}
$$\dim_{\ka}(\mathcal{B}(\ka))\geqslant e.$$
\begin{proof}[Proof of Claim 1]
Suppose not. Then, there is $\bar{b}\in \mathcal{B}(\ka)^{\times (e-1)}$ such that for the following group scheme morphism (defined above):
$$s_{\bar{b}}:\ga^{e-1}\ra \mathcal{B},$$
the group homomorphisms on $\ka$-rational points:
$$\left(s_{\bar{b}}\right)_{\ka}:\ga^{e-1}(\ka)\ra \mathcal{B}(\ka)$$
is onto. Since $(\mathcal{B},+)=\ga^{e}$ , it is impossible (it is clear for a finite $\ka$; if $\ka$ is infinite, then we get a dominant morphism $\Aa^{e-1}_{\ka}\to \Aa^{e}_{\ka}$ giving a contradiction).
\end{proof}
\noindent
{\bf Claim 2}
$$\dim_{\ka}(\mathcal{B}(\ka))\leqslant e.$$
\begin{proof}[Proof of Claim 2]
Suppose not. Then, there is a tuple $\bar{b}\in \mathcal{B}(\ka)^{\times(e+1)}$ such that the homomorphism
on $\ka$-rational points:
$$s:=\left(s_{\bar{b}}\right)_{\ka}:\ga^{e+1}(\ka)\ra \mathcal{B}(\ka)$$
is one-to-one. As in the proof of Lemma \ref{diagonal}, $s$ is given by an $e\times (e+1)$ matrix $M$ with coefficients in $\ka[\fr]$ and we can apply the
``Gaussian elimination'' over $\ka[\fr]$ to $M$. Since the matrix $M$ has less rows than columns,
we can transform $M$ into a matrix $M'$ having at least one zero column. Since the map $s$ corresponding to $M$ is one-to-one, the map corresponding to $M'$ is one-to-one as well, a contradiction.
\end{proof}
\begin{remark}
The perfectness assumption in Theorem \ref{dime} was used only for the inequality in Claim 2 above, however, this assumption can not be dropped. For example, if $\ka$ is a non-perfect field of characteristic $p>0$, then for $\lambda \in \ka\setminus \ka^p$ and the group scheme morphism:
$$\Psi:\ga^2\ra \ga,\ \ \ \ \ \ \ \ \ \  \Psi(x,y)=x^p+\lambda y^p,$$
we get that the map $\Psi_{\ka}:\ga^2(\ka)\to \ga(\ka)$ is one-to-one.
\end{remark}
We introduce below a particular morphism of $\ka$-algebra schemes, which will play the role of $\phi$ from Lemma \ref{diagonaltofr}.
\begin{notation}\label{nottheta}
We assume that $\mathcal{B}$ is a $\ka$-algebra scheme such that $(\mathcal{B},+)=\ga^e$ and $\ka$ is a perfect field. By Theorem \ref{dime}, we have $\dim_{\ka}\mathcal{B}(\ka)=e$, so $\mathcal{B}(\ka)_\otimes$ is a $\ka$-algebra scheme by Example \ref{example1}. Then, there is a natural transformation:
$$\Theta:\mathcal{B}(\ka)_{\otimes}\ra \mathcal{B},\ \ \ \ \ \ \ \ \ \ \ \ \ \Theta_R:\mathcal{B}(\ka)\otimes_{\ka}R\to \mathcal{B}(R);$$
where for a $\ka$-algebra $R$, the map $\Theta_R$ is given by the universal property of the tensor product and the $\ka$-algebra homomorphisms $\mathcal{B}\left( \ka \right) \to \mathcal{B}\left( R \right),R \to\mathcal{B}\left( R \right)$.
\end{notation}
We are ready now to show our main classification result.
\begin{theorem}\label{classop}
Assume that $\ka$ is perfect and $\mathcal{B}$ is a $\ka$-algebra scheme such that $(\mathcal{B},+)=\ga^e$. Then, up to an isomorphism, for some $n_1,\ldots,n_e\in \Nn$ we have:
$$\Theta = \left(\fr^{n_1},\fr^{n_2},\dotsc ,\fr^{n_e}\right),$$
where $\Theta$ is the $\ka$-algebra scheme morphism introduced in Notation \ref{nottheta}.

Moreover, if $\mathcal{B}$ is a coordinate $\ka$-algebra scheme, then $n_1=0$.
\end{theorem}
\begin{proof}
By applying Lemma \ref{diagonaltofr} to the morphism $\Theta$, we almost get the desired form of $\Theta$ except that it is still possible to have the zero morphism on some coordinate. However, $\Theta_{\ka}$ is the identity map on $\mathcal{B}(\ka)$, hence no zero morphism can appear.

For the moreover part, we notice that if $\mathcal{B}$ is a coordinate $\ka$-algebra scheme, then $\Theta$ commutes with the projection on the first coordinate (since $\pi$ is a natural map between $\mathcal{B}$ and the identity functor) implying $n_1=0$.
\end{proof}
We immediately get the following consequence of Theorem \ref{classop} saying that in the case of characteristic $0$, there are no ``new'' $\ka$-algebra schemes.
\begin{cor}\label{charac0}
Suppose that $\mathrm{char}(\ka)=0$ and $\mathcal{B}$ is as in Theorem \ref{classop}. Then, we have the following isomorphism of $\ka$-algebra schemes:
$$\mathcal{B}\cong \mathcal{B}(\ka)_{\otimes}.$$
\end{cor}
We also obtain the following.
\begin{cor}\label{perext}
Let $\ka,\mathcal{B}$ be as in Theorem \ref{classop} and suppose that $\ka\subseteq K$ is an extension of perfect fields. Then, we have the following isomorphism of $K$-algebras:
$$\mathcal{B}(K)\cong \mathcal{B}(\ka)\otimes_{\ka}K.$$
\end{cor}
\begin{proof}
By Theorem \ref{classop}, the $K$-algebra map
$$\Theta_K:\mathcal{B}(\ka)\otimes_{\ka}K\to \mathcal{B}(K)$$
(see Notation \ref{nottheta}) is one-to-one. Since both the fields $\ka$ and $K$ are perfect, by Theorem \ref{dime} we get:
$$\dim_{\ka}(\mathcal{B}(\ka))=e=\dim_{K}(\mathcal{B}(K)),$$
hence the map $\Theta_K$ is an isomorphism.
\end{proof}

\begin{remark}
We briefly discuss the case of a non-perfect field $\ka$.
We observe that there is a finite purely inseparable field extension $\ka\subseteq \ka'$ such that Theorem \ref{classop} holds for $\mathcal{B}_{\ka'}$, since Theorem \ref{classop} holds for $\mathcal{B}_{\ka^{\mathrm{ins}}}$, where $\ka^{\mathrm{ins}}$ is the inseparable closure of $\ka$, but $\mathcal{B}(\ka^{\mathrm{ins}})$ is defined over a finite extension of $\ka$.
\end{remark}
We are able now to give algebraic conditions explaining when a $\ka$-algebra scheme is of the form $B_{\otimes}$. Because of Corollary \ref{charac0}, it is natural to assume that the characteristic of $\ka$ is positive.
\begin{prop}\label{eqcon}
Assume that $\ka$ is a perfect field of positive characteristic and $\mathcal{B}$ is a $\ka$-algebra scheme such that $(\mathcal{B},+)=\ga^e$. Then, the following are equivalent.
\begin{enumerate}
\item The natural map $\Theta$ from Notation \ref{nottheta} is an isomorphism.

\item The functor $\mathcal{B}$ is isomorphic (as a $\ka$-algebra scheme) to the functor $B_{\otimes}$ as in Example \ref{example1} for some $\ka$-algebra $B$ (then, necessarily $B\cong \mathcal{B}(\ka)$).

\item There is a ``$\ka$-module scheme'' isomorphism:
$$\mathcal{B}\cong \left(\mathbb{S}_{\ka}\right)^{\times e}.$$
In particular, for each $\ka$-algebra $R$, we have $\mathcal{B}(R)\cong R^{\times e}$ as $R$-modules.

\item For any field extension $\ka\subseteq K$, we have:
$$\dim_K(\mathcal{B}(K))=e.$$

\item There is a field extension $\ka\subseteq K$ such that $K$ is not perfect and:
$$\dim_K(\mathcal{B}(K))=e.$$
\end{enumerate}
\end{prop}
\begin{proof}
The implications $(1)\Rightarrow (2)\Rightarrow (3)\Rightarrow (4)\Rightarrow (5)$ are obvious.

To show the implication $(5)\Rightarrow (1)$, we use Theorem \ref{classop} and obtain that up to an isomorphism, we have:
$$\Theta=\left(\fr^{n_1},\fr^{n_2},\ldots,\fr^{n_e}\right)$$
for some $n_1,\ldots,n_e\in \Nn$. We need to show that under the assumption of Item $(5)$, we have $n_i=0$ for all $i$. We note first that for each field extension $\ka\subseteq K$, the map
$$\Theta_K=\left(\fr^{n_1}_K,\fr^{n_2}_K,\ldots,\fr^{n_e}_K\right):K^{\times e}\to \mathcal{B}(K)$$
is a $K$-linear isomorphism, where additively $(\mathcal{B}(K),+)=(K^{\times e},+)$. It follows that there is the following isomorphism of vector spaces over $K$:
$$\mathcal{B}(K)\cong K\times \left(\fr^{n_2}_K\right)_*(K)\times \ldots \left(\fr^{n_e}_K\right)_*(K),$$
where for any field endomorphism $\rho:K\to K$, $\rho_*(K)$ denotes the $K$-vector space $(K,+)$ with the scalar multiplication given by:
$$a*x:=\rho(a)x.$$
Therefore, if $n_i\neq 0$ for some $i$ and $K$ is not perfect, then we get
$$\dim_K(\mathcal{B}(K))>e,$$
which contradicts the assumption from Item $(5)$.
\end{proof}

\section{$\mathcal{B}$-operators}\label{secoper}
In this section, we introduce our notion of operators on rings and we discuss the algebraic properties of these operators. Afterwards, we define prolongations and prove the necessary results about them, which will be crucial for the model-theoretic considerations of the next section.
Let $(\mathcal{B},\iota, \pi)$ be a fixed coordinate $\ka$-algebra scheme (see Definition \ref{bdef}) such that $(\mathcal{B},+)=\ga^e$.

\subsection{Some algebra of $\mathcal{B}$-operators}
We introduce below the main notion of this paper. We are interested in operators going from a ring to the same ring, but, for technical reasons, we also need to consider operators between two different rings.
\begin{definition}\label{bopdef}
Let $R$ and $T$ be $\ka$-algebras and $f:R\to T$ be a $\ka$-algebra map.
\begin{enumerate}
\item A \emph{$\mathcal{B}$-operator on $R$} is a $\ka$-algebra map $\partial:R\to \mathcal{B}(R)$ such that $\pi_R\circ \partial=\id$.

\item More generally, a \emph{$\mathcal{B}$-operator from $R$ to $T$} (\emph{of $f$}) is a $\ka$-algebra map $\partial:R\to \mathcal{B}(T)$ (such that $\pi_T\circ \partial=f$).

\item For a $\mathcal{B}$-operator $\partial:R \to T$, the \emph{ring of constants} of $\partial$ is defined as:
$$R^{\partial}:=\{r\in R\ |\ \partial(r)=\iota_T(\pi_T(\partial(r)))\}.$$
(Clearly, if $\partial$ is a $\mathcal{B}$-operator of $f$, then we also have:
$$R^{\partial}:=\{r\in R\ |\ \partial(r)=\iota_T(f(r))\}.)$$

\item A field with a $\mathcal{B}$-operator is called a \emph{$\mathcal{B}$-field}. Similarly, we define \emph{$\mathcal{B}$-rings}, \emph{$\mathcal{B}$-field extensions}, etc.
\end{enumerate}
\end{definition}
\begin{remark}\label{mscontext}
Let $R$ and $T$ be $\ka$-algebras.
\begin{enumerate}
\item Assume that $\mathcal{B}=B_{\otimes}$ for some $B$ as in Example \ref{firstex}(1). In this special case, the notion of a $\mathcal{B}$-operator coincides with the notion of a $B$-operator from \cite{BHKK}, which we quickly recall below.

Let $\{b_1,\ldots,b_e\}$ be a fixed $\ka$-basis of $B$ such that $\pi_B(b_1)=1$. For any $\ka$-algebra $R$ and any tuple $\partial=(\partial_1,\ldots,\partial_e)$ where $\partial_i:R\to R$, we say that $\partial$ is a \emph{$B$-operator} if the corresponding map
$$R\ni r\mapsto \partial_1(r)\otimes b_1+\ldots+\partial_e(r)\otimes b_e\in R\otimes_{\ka}B$$
is a $\ka$-algebra map such that $\partial_1=\id$.

\item Since $(\mathcal{B},+)=\ga^e$, any $\mathcal{B}$-operator from $R$ to $T$ corresponds to a certain sequence $\partial=(\partial_1,\ldots,\partial_e)$ of maps $R\to T$. There are polynomials $F_1,\ldots,F_e\in \ka[X_1,Y_1,\ldots,X_e,Y_e]$  such that $\partial$ is a $\mathcal{B}$-operator if and only if for each $i$, the map $\partial_i$ is additive and for each $x,y\in R$, we have:
$$\partial_i(xy)=F_i\left(\partial_1(x),\partial_1(y),\ldots,\partial_e(x),\partial_e(y)\right).$$
From the description above, there is a clear relation between our $\mathcal{B}$-operators and Buium's jet operators from \cite{Bui3}. If $\ka$ is a field, then
(by Example \ref{firstex}(3)) jet operators coincide with $\mathcal{B}$-operators for $e=2$. If $\ka$ is an arbitrary ring, then a jet operator need not be additive, for example $\pi$-derivations from \cite[Example(d)]{Bui3} are not additive.
\end{enumerate}
\end{remark}
\begin{example}\label{bopex}
We fix a $\ka$-algebra $R$ and give a few examples of $\mathcal{B}$-operators on $R$.
\begin{enumerate}
\item There is always the ``zero $\mathcal{B}$-operator'' on $R$ given by the structure homomorphism $\iota_R$.

\item By Remark \ref{mscontext}(1), any $B$-operator is a $\mathcal{B}$-operator for $B:=\mathcal{B}(\ka)$. We mention below two most standard examples of such $\mathcal{B}$-operators.
\begin{enumerate}
\item If $B=\ka^{\times e}$, then $(\id,\partial_2,\ldots,\partial_e)$ is a $B$-operator on $R$ if and only if $\partial_2,\ldots,\partial_e$ are $\ka$-algebra endomorphisms of $R$.

\item If $B=\ka[X]/(X^2)$, then $(\id,\partial)$ is a $B$-operator on $R$ if and only if $\partial$ is a derivation on $R$ vanishing on $\ka$.
\end{enumerate}

\item The notion of a $\mathcal{B}$-operator covers the case of \emph{derivations of Frobenius}, which do not fit to the set-up from \cite{MS2}. Suppose that $\ka=\Ff_p$. An additive map $\partial$ on $R$ is a derivation of Frobenius, if for all $x,y\in R$, we have:
    $$\partial(xy)=x^p\partial(y)+y^p\partial(y).$$
    It is clear that an arbitrary map $\partial:R\to R$ is a derivation of Frobenius if and only if the tuple $(\id,\partial)$
is a $\mathcal{B}$-operator for $\mathcal{B}$ as in Example \ref{firstex}(2).

More generally, if $\mathrm{char}(\ka)=p>0$, then we can consider
$$\mathcal{B}:=\left(\left(\ka[X]/\left(X^2\right)\right)_{\otimes}\right)^{\fr}$$
as in Example \ref{trexam}(1). Then, a tuple $(\id,\partial)$ is a $\mathcal{B}$-operator if and only if $\partial$ is a \emph{derivations of Frobenius over $\ka$} that is $\partial$ is a derivation of Frobenius and moreover for any $a\in \ka$ and $r\in R$, it satisfies:
$$\partial(ax)=a^p\partial(x).$$

\item Generalizing the case of Item $(3)$ above, we can consider ``$B$-operators of Frobenius'' by replacing the ring of dual numbers $\ka[X]/(X^2)$ with any finite $\ka$-algebra $B$ as in Example \ref{firstex}(1) and considering $\mathcal{B}$-operators, for:
    $$\mathcal{B}:=\left(B_{\otimes}\right)^{\fr}.$$
\end{enumerate}
\end{example}


Let us consider now the kernel of $\pi$ as the following ``scheme of ideals'':
$$\ker(\pi)(R):=\ker\left(\pi_R:\mathcal{B}(R)\ra R\right).$$
We state below an important assumption on $\mathcal{B}$, which we will often make. From now on, we focus on the case of positive characteristic, since (by Corollary \ref{charac0}) if $\mathrm{char}(\ka)=0$, then $\mathcal{B}$-operators are the same as $\mathcal{B}(\ka)$-operators, which were analyzed in \cite{MS2}.  We denote the Frobenius morphism on $\mathcal{B}$ by $\fr_{\mathcal{B}}$.
\begin{assumption}\label{ass2}
Let $\mathrm{char}(\ka)=p>0$ and:
$$\fr_{\mathcal{B}}\left(\ker(\pi)\right)=0.$$
\end{assumption}
\begin{remark}\label{assrem}
We discuss here some consequences of Assumption \ref{ass2}.
\begin{enumerate}

\item In the case of $\mathcal{B}=B_{\otimes}$, Assumption \ref{ass2} is equivalent to:
$$\fr_{B}\left(\ker(\pi_B)\right)=0$$
and its meaning was discussed in \cite[Lemma 2.6]{BHKK}, which says that it is equivalent to $B$ being local and such that the nilradical of $B$ coincides with $\ker(\fr_B)$. It was also noticed in \cite[Lemma 2.6]{BHKK} that $B$ is local if and only if $\ker(\pi_B)^e=0$ (the reader is advised to recall our notational conventions from the beginning of Section \ref{secrings}).


\item The condition ``$\ker(\pi)^e=0$'' makes sense for ring schemes as well and it means that the $e$-fold multiplication morphism:
$$m_e:\mathcal{B}^{\times e}\ra \mathcal{B}$$
vanishes on the subscheme $\ker(\pi)^{\times e}$. Since this condition can be checked on any infinite field extending $\ka$, Item $(1)$ above yields that Assumption \ref{ass2} implies the condition $\ker(\pi)^{e}=0$, similarly as in the case of $\mathcal{B}=B_{\otimes}$.

\end{enumerate}
\end{remark}
All the remaining results of this subsection originate from \cite{BHKK}, where the case of $\mathcal{B}=B_{\otimes}$ (in the terminology of this paper) was considered. We attempt to present these results here (in the more general context of an arbitrary $\mathcal{B}$) in a most economical way without aiming for the greatest generality possible. Therefore, at the same time we generalize and simplify some of the results and proofs from \cite{BHKK}.
\begin{lemma}\label{ver27}
Suppose $R,S$ are $\ka$-algebras,  Assumption \ref{ass2} holds for $\mathcal{B}$, and $\partial$ is a $\mathcal{B}$-operator from $R$ to $S$ such that the map
$$\partial_0:=\pi_S\circ \partial:R\ra S$$
is an embedding. Then, we have the following.
\begin{enumerate}
\item If $R$ and $S$ are domains, then $\partial$ extends uniquely to a $\mathcal{B}$-operator on their fields of fractions.

\item If $\partial_0:R\to S$ is an \'{e}tale extension of rings, then $\partial$ extends uniquely to a $\mathcal{B}$-operator on $S$.

\item If $\partial_0:R\to S$ is a formally smooth (or $0$-smooth) extension of rings, then $\partial$ extends to a $\mathcal{B}$-operator on $S$.
\end{enumerate}
\end{lemma}
\begin{proof}
In the context of $B$-operators, this result appeared as \cite[Lemma 2.7]{BHKK} and the arguments from there work in our case as well. We still give below a proof of Item $(1)$ to make clear how Assumption \ref{ass2} is used.

We have the following commutative diagram:
\begin{equation*}
 \xymatrix{  R \ar[d]_{}  \ar@/^2pc/[rr]^{\partial^{\mathrm{Frac}}} \ar[r]^-{\partial} \ar[rd]^-{\partial_0} &  \mathcal{B}(S)  \ar[d]_{} \ar[r]^{} &   \mathcal{B}(\mathrm{Frac}(S)) \ar[d]_{\pi} \\
\mathrm{Frac}(R)  \ar@/_2pc/[rr]^{\mathrm{Frac}(\partial_0)} &  S \ar[r]^{} & \mathrm{Frac}(S),}
\end{equation*}
where the map $\mathrm{Frac}(\partial_0)$ exists, since $\partial_0$ is an embedding. We need to show that there is a unique map
$$\widetilde{\partial}:\mathrm{Frac}(R)\ra \mathcal{B}(\mathrm{Frac}(S))$$
such that the above diagram expanded by this map is still commutative. We simplify the above diagram to the following one:
\begin{equation*}
 \xymatrix{  R \ar[d]_{}   \ar[r]^-{\partial^{\mathrm{Frac}}}  &   \mathcal{B}(\mathrm{Frac}(S)) \ar[d]_{\pi} \\
\mathrm{Frac}(R) \ar@{-->}[ru]^{\widetilde{\partial}}  \ar[r]^{} & \mathrm{Frac}(S).}
\end{equation*}
By Remark \ref{assrem}(2), we get that $\ker(\pi)^e=0$. Since the map $R\to \mathrm{Frac}(R)$ is \'{e}tale (note that in the definition of  \'{e}tality from e.g. \cite{mat} one can replace ``2-nilpotent'' with ``$e$-nilpotent'' and obtain an equivalent definition) and
$$\ker\left(\pi:\mathcal{B}(\mathrm{Frac}(S))\to \mathrm{Frac}(S)\right)^e=0$$
(this is where we use Assumption \ref{ass2}, or rather its consequence from Remark \ref{assrem}(2)), we get that there is a unique map $\widetilde{\partial}$ completing this last diagram.
\end{proof}
\begin{lemma}\label{ver28}
Suppose that Assumption \ref{ass2} is satisfied and let $\partial:R\to \mathcal{B}(T)$ be a $B$-operator.
Then, we have:
$$R^p\subseteq R^{\partial}.$$
\end{lemma}
\begin{proof}
Since $\fr_{\mathcal{B}(T)}(\ker(\pi_T))=0$, the following diagram commutes:
\begin{equation*}
 \xymatrix{  R  \ar[d]_{\partial} \ar[rr]^{\fr_R} &  &  R  \ar[d]^{\partial} \\
  \mathcal{B}(T)  \ar[d]_{\pi_T} \ar[rr]^{\fr_{\mathcal{B}(T)}} &  & \mathcal{B}(T) \\
T\ar[rr]^{\fr_T}  &  &  T \ar[u]^{\iota_T},}
\end{equation*}
which finishes the proof.
\end{proof}
\begin{lemma}\label{ver29}
Let $\ka\subseteq M\subseteq N$ be a tower of fields and $\partial$ be a $\mathcal{B}$-operator from $M$ to $N$ (of the inclusion $M\subseteq N$) such that $N^p\subseteq M^{\partial}$. Then, $\partial$ extends to a $\mathcal{B}$-operator on $N$.
\end{lemma}
\begin{proof}
Since $N^p\subseteq M$, $N$ is a purely inseparable extension of $M$, and we may assume by induction that  $M=M(t^{1/p})$ for some $t\in M$. Since $t\in N^p$, by our assumption we get that $t\in M^{\partial}$, hence $\partial(t)=\iota_M(t)$ (see the second displayed line in Definition \ref{bopdef}(3)). Since we have:
$$\iota_N(t^{1/p})^p=\iota_N(t)=\iota_M(t)=\partial(t),$$
we may set $\partial_N(t^{1/p}):=\iota_N(t^{1/p})$ to obtain a $\mathcal{B}$-operator on $N$ extending $\partial$.
\end{proof}
The crucial application of Assumption \ref{ass2} is contained in the statement below.
\begin{prop}\label{ver215}
Suppose that Assumption \ref{ass2} holds,  $\ka\subseteq M\subseteq N\subseteq \Omega$ is a tower of fields, and we have the following commuting diagram of $\ka$-algebra maps:
\begin{equation*}
 \xymatrix{  M  \ar[d]_{} \ar[rr]^{\partial}  &  &  \mathcal{B}(N)  \ar[d]^{} \\
N  \ar[rrd]_{} \ar[rr]^{\partial_{\Omega}}  &  &  \mathcal{B}(\Omega)  \ar[d]^{\pi_{\Omega}}\\
 &  &  \Omega,}
\end{equation*}
that is: the $\mathcal{B}$-operator $\partial$ from $M$ to $N$ extends to a $\mathcal{B}$-operator $\partial_{\Omega}$ of the inclusion $N\subseteq \Omega$. Then, $\partial$ extends to a $\mathcal{B}$-operator $\partial_N$ on $N$ as well.
\end{prop}
\begin{proof}
There is a subfield $N_0\subseteq N$ such that the field extension $M\subseteq N_0$ is separable (not necessarily algebraic) and the extension $N_0\subseteq N$ is purely inseparable. By Lemma \ref{ver27}(3) (a separable field extension is formally smooth, see \cite[Theorem 26.9]{mat}), without loss of generality we can assume that $M=N_0$. By Lemma \ref{ver28}, we have $N^p\subseteq N^{\partial_{\Omega}}$.
Since $N^p\subseteq M$ and $\partial_{\Omega}$ extends $\partial$, we get that $N^p\subseteq M^{\partial}$. Applying Lemma \ref{ver29}, we obtain the required $\mathcal{B}$-operator $\partial_N$.
\end{proof}
Proceeding exactly as in \cite[Section 2.3]{BHKK} (after replacing the functor $B_{\otimes}$ with the functor $\mathcal{B}$), we obtain the following.
\begin{prop}\label{ver220}
If $K$ is a $\mathcal{B}$-field and $K\subseteq R,K\subseteq S$ are $\mathcal{B}$-ring extensions, then there is a unique $\mathcal{B}$-operator on the ring $R\otimes_KS$ such that the natural maps $R\to R\otimes_KS,S\to R\otimes_KS$ commute with the corresponding $\mathcal{B}$-operators.
\end{prop}

\subsection{Prolongations with respect to $\mathcal{B}$-operators}\label{secprol}
Our aim here is to construct left-adjoint functors to some functors related to our fixed coordinate $\ka$-algebra scheme $\mathcal{B}$ (regarded as an endofunctor on the category $\mathrm{Alg}_{\ka}$). We can not use the classical Weil restriction (as it was done in \cite{MS2}), since for some field extensions $\ka\subseteq K$, $\mathcal{B}(K)$ need not be a finite $K$-algebra and the criterion from \cite[Section 7.6, Theorem 4]{neron} is not applicable. Therefore, we need to construct these left-adjoint functors ``by hand''. Actually, the standard construction of the Weil restriction functor works in our more general context as well.

Let us fix a field extension $\ka\subseteq K$ and a $\mathcal{B}$-operator $\partial$ on $K$. For any affine scheme $V$ over $K$,
we want to define its prolongation $\tau^{\partial}(V)$. The defining property of $\tau^{\partial}(V)$ is that for any $K$-algebra $R$, we should have a natural bijection:
$$\tau^{\partial}(V)(R)\longleftrightarrow V\left(\mathcal{B}(R)\right),$$
where $\mathcal{B}(R)$ has the $K$-algebra structure given by the composition of ${\partial}$ with the map $\mathcal{B}(K)\to \mathcal{B}(R)$. Since we are interested in affine varieties only, we are in fact looking for a left-adjoint functor to a specialization of the functor $\mathcal{B}$.

Let us describe the aforementioned specialization of $\mathcal{B}$ in a more detailed way. For any $K$-algebra $R$, we set $\mathcal{B}^{\partial}(R):=\mathcal{B}(R)$ as rings. To define a $K$-algebra structure on $\mathcal{B}^{\partial}(R)$, we consider the composition of the following two maps:
\begin{equation*}
 \xymatrix{K\ar[rr]^-{\partial} &  & \mathcal{B}(K) \ar[rr]^-{\mathcal{B}(\rho)} &  & \mathcal{B}(R),}
\end{equation*}
where $\rho:K\to R$ is the structure map. It is easy to see that for any $K$-algebra map $f:R\to S$, the induced map $\mathcal{B}(f):\mathcal{B}(R)\to \mathcal{B}(S)$ is a $\mathcal{B}(K)$-algebra map. Hence, it is also a $K$-algebra map $\mathcal{B}^{\partial}(R)\to \mathcal{B}^{\partial}(S)$ and we get a functor:
$$\mathcal{B}^{\partial}:\Alg_K \ra \Alg_K.$$
\begin{theorem}\label{prolexist}
The functor $\mathcal{B}^{\partial}$ has a left-adjoint functor.
\end{theorem}
\begin{proof}
We will construct a functor
$$\tau^{\partial}:\Alg_K \ra \Alg_K$$
such that for any $K$-algebras $R,T$, there is a natural bijection:
$$\Hom_{\Alg_K}\left(R,\mathcal{B}^{\partial}(T)\right)\longleftrightarrow \Hom_{\Alg_K}\left(\tau^{\partial}(R),T\right).$$
For a (possibly infinite) tuple of variables $X$, let $\bar{X}:=(X_1,\ldots,X_e)$ denote a new tuple of variables, where each $X_i$ has the same length as $X$ and we make the identification $X=X_1$. Firstly, we will extend $\partial$ to a $\mathcal{B}$-operator $\bar{\partial}$ from $K[X]$ to $K[\bar{X}]$. Since we have a natural bijection:
$$\mathcal{B}\left(K\left[\bar{X}\right]\right)\longleftrightarrow K\left[\bar{X}\right]^{\times e},$$
we can consider $\bar{X}$ as an element of $\mathcal{B}\left(K[\bar{X}]\right)$ and we define $\bar{\partial}(X):=\bar{X}$. Together with $\partial$ on $K$, this determines a $K$-algebra homomorphism
$$\bar{\partial}:K[X]\ra \mathcal{B}\left(K\left[\bar{X}\right]\right),$$
which is a $\mathcal{B}$-operator from $K[X]$ to $K[\bar{X}]$, so it can be identified with a tuple $(\partial_1,\ldots,\partial_e)$, where each $\partial_i$ is a map from $K[X]$ to $K[\bar{X}]$ (see Remark \ref{mscontext}(2)).

We can assume that $R=K[X]/I$ for an appropriate $X$ as above and an ideal $I$ of $K[X]$. We define our adjoint functor by setting:
$$K[X]/I\mapsto K\left[\bar{X}\right]/\bar{I},\ \ \ \ \ \ \bar{I}:=\left(\partial_1(I)\cup \ldots\cup \partial_e(I)\right).$$
We will check that this construction gives the desired adjunction. That is, we need to find the following natural bijection:
$$\Hom_{\Alg_K}\left(K[X]/I,\mathcal{B}^{\partial}(T)\right)\longleftrightarrow \Hom_{\Alg_K}\left(K\left[\bar{X}\right]/\bar{I},T\right).$$
Since we have a $K$-scheme identification $\mathcal{B}^{\partial}=\Aa^e_K$, we get that the $K$-algebra maps $\Psi:K[X]\to \mathcal{B}^{\partial}(T)$ are in a natural bijection with the $K$-algebra maps $\bar{\Psi}:K[\bar{X}]\to T$, where $\bar{\Psi}(X_i)=\Psi_i(X)$ (after the identification $\Psi=(\Psi_1,\ldots,\Psi_e)$, where $\Psi_i:K[X]\to T$). It is easy to see (by checking the commutativity on $X$) that we have the following commutative diagram:
\begin{equation*}
 \xymatrix{  K[X]  \ar[d]_{\bar{\partial}} \ar[rr]^{\Psi}   &  & \mathcal{B}^{\partial}(T)  \\
\mathcal{B}^{\partial}(K\left[\bar{X}\right])   \ar[rru]_{\mathcal{B}^{\partial}\left(\bar{\Psi}\right)},  &  &  }
\end{equation*}
which implies that for all $f\in K[X]$, we have:
$$\Psi(f)=0\ \ \ \ \ \Longleftrightarrow \ \ \ \ \ \ \forall i\ \ \bar{\Psi}\left( \partial_i(f) \right)=0.$$
In particular, we get that:
$$I\subseteq \ker(\Psi)\ \ \ \ \ \Longleftrightarrow \ \ \ \ \ \ \bar{I}\subseteq \ker\left(\bar{\Psi}\right).$$
Hence, the bijection $\Psi\leftrightarrow \bar{\Psi}$ extends to our adjointness bijection.
\end{proof}
\begin{remark}
For the above construction, we do not need all the data from the definition of a coordinate $\ka$-algebra scheme. It is enough to assume (restricting to the category of $K$-algebras) that $\mathcal{B}$ is a ring scheme over $K$ whose underlying scheme is $\Aa^e_K$.
\end{remark}
For a $K$-algebra $R$, we describe now a natural homomorphism of $K$-algebras
$$\pi^{\partial}:R\ra \tau^{\partial}(R),$$
whose properties will be important in the sequel. The construction of this homomorphism uses the morphism $\pi$ from the coordinate $\ka$-algebra scheme data.
Consider the adjointness bijection:
$$\Hom_{\Alg_K}\left(\tau^{\partial}(R),\tau^{\partial}(R)\right)\ra \Hom_{\Alg_K}\left(R,\mathcal{B}^{\partial}(\tau^{\partial}(R))\right),\ \ \ \ \ \ \ \ \
f\mapsto f^{\sharp}.$$
Then, $\pi^{\partial}$ is defined as the composition of the following maps:
\begin{equation*}
 \xymatrix{R\ar[rr]^-{\id^{\sharp}} &  & \mathcal{B}^{\partial}(\tau^{\partial}(R)) \ar[rr]^-{\pi_{\tau^{\partial}(R)}} &  & \tau^{\partial}(R).}
\end{equation*}
\begin{definition}
Let $V=\spec(R)$ be an affine $K$-scheme.
\begin{enumerate}
\item We define the \emph{$\partial$-prolongation} of $V$ as
$$\tau^{\partial}V:=\spec\left(\tau^{\partial}R\right).$$

\item We have a natural morphism
$$\pi^{\partial}:\tau^{\partial}V\ra V,$$
which was described above on the level of $K$-algebras.

\item  For any $K$-algebra $T$ and any $a\in V(T)$, we denote by $\tau^{\partial}_aV$ the scheme over $T$, which is the fiber of the morphism $\pi^{\partial}:\tau^{\partial}V\to V$ over $a$.
\end{enumerate}
\end{definition}
We would like to point out below an interpretation of rational points of prolongations in terms of $\mathcal{B}$-operators.
\begin{remark}\label{nrem}
If $R$ is a $K$-algebra and $V$ is an affine $K$-scheme, then (by adjointness) the set of $R$-rational points $\tau^{\partial}V(R)$ is in a natural bijection with the set of $\mathcal{B}$-operators from $K[V]$ to $R$ extending $\partial$. More precisely, for any such rational point $a\in \tau^{\partial}V(R)$, we get the corresponding rational point $\pi^{\partial}(a)\in V(R)$ and then $a$ corresponds to a $\mathcal{B}$-operator (extending $\partial$) from $K[V]$ to $R$ of the $\ka$-algebra map $K[V]\to R$ corresponding to $\pi^{\partial}(a)$. Moreover, if $b\in V(R)$, then the fiber $\tau^{\partial}_bV(R)$ (considered as a subset of $\tau^{\partial}V(R)$) corresponds to $\mathcal{B}$-operators $\widetilde{\partial}:K[V]\to \mathcal{B}(R)$ such that $\pi_R\circ \widetilde{\partial}=b$, where $b$ is considered as $K$-algebra homomorphism $K[V]\to R$.
\end{remark}
In particular,  we get the following.
\begin{remark}\label{ver213}
There is a natural (in $V$) map (not a morphism!):
$$\partial_V:V(K)\ra \tau^{\partial}V(K)=V\left(\mathcal{B}^{\partial}(K)\right)$$
coming from composing the map $x:K[V]\to K$ with $\partial$ (equivalently, $\partial_V=V(\partial)$).
\end{remark}
\begin{lemma}\label{ver214}
Suppose that $V$ and $W$ are affine $K$-schemes and $W\subseteq \tau^{\partial}(V)$. Then, we get a natural $\mathcal{B}$-operator
$$\partial^W_V:K[V]\ra \mathcal{B}(K[W])$$
of the corresponding map $K[V]\to K[W]$ (see Remark \ref{nrem}) extending $\partial$.
\end{lemma}
\begin{proof}
The inclusion $W\subseteq \tau^{\partial}(V)$ gives a rational point $\partial^W_V\in \tau^{\partial}(V)(K[W])$, which corresponds to the desired $\mathcal{B}$-operator extending $\partial$ as it was observed above.
\end{proof}
We will show now generalizations of two results from \cite{BHKK} regarding rational points of fibers of prolongations under Assumption \ref{ass2}. We will need only the second result. However, this second result may look a bit mysterious, hence we prefer to include the first one as well, whose intuitive meaning is much clearer.

By a \emph{$K$-variety}, we mean an affine $K$-scheme, which is of finite type and $K$-irreducible (not necessarily absolutely irreducible!).
\begin{cor}\label{ver216}
We work under Assumption \ref{ass2} and assume moreover that:
\begin{itemize}
\item $\ka\subseteq K\subseteq \Omega$ is a tower of fields;

\item $\partial$ is a $\mathcal{B}$-operator on $K$;

\item $W$ is a $K$-variety;

\item there is $c\in \tau^{\partial}W(\Omega)$ such that
$$b:=\pi^{\partial}(c)\in W(\Omega)$$
is a generic point of $W$ over $K$.
\end{itemize}
Then $\tau^{\partial}_bW(K(b))\neq \emptyset$.
\end{cor}
\begin{proof}
We consider $b\in W(\Omega)$ as a $K$-algebra homomorphism $b:K[W]\to \Omega$ and $c\in \tau^{\partial}W(\Omega)$ as a $B$-operator from $K[W]$ to $B$ of $b$, so we get the following commutative diagram resembling the diagram from the assumptions of Proposition \ref{ver215}:
\begin{equation*}
 \xymatrix{  K  \ar[d]_{} \ar[rr]^{\partial}  &  &  \mathcal{B}(K)  \ar[d]^{} \\
K[W]  \ar[rrd]_{b} \ar[rr]^{c}  &  &  \mathcal{B}(\Omega)  \ar[d]^{\pi_{\Omega}}\\
 &  &  \Omega.}
\end{equation*}
Let $L:=K(W)$ be the field of fractions of $K[W]$. By Lemma \ref{ver27}(1), the $\mathcal{B}$-operator $c$ extends to $L$.
Since $b$ is a generic point of $W$ over $K$, the corresponding map $b:K[W]\to \Omega$ is one-to-one and it extends to $L$. If we regard this map $L\to \Omega$ as an inclusion and consider $\partial$ as a $\mathcal{B}$-operator from $K$ to $L$, then we are in the situation from Proposition \ref{ver215} (for $M=K$ and $N=L=K(b)$) and we get the corresponding $\mathcal{B}$-operator on $K(b)$ extending $\partial$, which is an element of $\tau^{\partial}_bW(K(b))$.
\end{proof}
Following \cite{BHKK}, for any pair of affine $K$-schemes $(V,W)$ such that $W\subseteq \tau^{\partial}(V)$, we define the $K$-scheme $E$ as the equalizer of the following two morphisms:
$$\tau^{\partial}\left(\pi_V^{\partial}\circ i\right),i \circ \pi_W^{\partial}:\tau^{\partial}(W)\ra \tau^{\partial}(V),$$
where $i$ is the inclusion morphism $W\subseteq \tau^{\partial}(V)$. We also denote by $\pi_E$ the following composition morphism:
\begin{equation*}
 \xymatrix{E\ar[r]^{} &  \tau^{\partial}(W) \ar[r]^{\ \ \pi_{W}^{\partial}} &   W.}
\end{equation*}
The above definition of the $K$-scheme $E$ may look strange and not very well motivated, but it is crucial for the geometric axioms given in the next section. We describe below the scheme $E$ in terms of its rational points (and the corresponding $\mathcal{B}$-operators) and this description should clarify the purpose of introducing $E$.
\begin{remark}\label{eclar}
Let us take a rational point $\widetilde{\partial}\in \tau^{\partial}(W)(\Omega)$ ($K\subseteq \Omega$ is a field extension). As explained above, $\widetilde{\partial}$ corresponds to the following $\mathcal{B}$-operator:
$$\widetilde{\partial}:K[W]\ra \mathcal{B}(\Omega).$$
This $\mathcal{B}$-operator yields \emph{two} $\mathcal{B}$-operators from $K[V]$ to $\Omega$ in the following way.
\begin{itemize}
\item The composition morphism $W\to \tau^{\partial}(V)\to V$ gives a $K$-algebra map $K[V]\to K[W]$ and we define
$$\widetilde{\partial}_V:K[V]\ra \mathcal{B}^{\partial}(\Omega)$$
as $\widetilde{\partial}$ composed with the above map.

\item Let us set as usual:
$$\widetilde{\partial}_0:K[W]\ra \Omega,\ \ \ \ \ \ \ \ \ \ \widetilde{\partial}_0:=\pi\circ \widetilde{\partial}.$$
Then, we define
$$\widetilde{\partial}^V:K[V]\ra \mathcal{B}^{\partial}(\Omega)$$
as the following composition:
\begin{equation*}
 \xymatrix{K[V]\ar[rr]^-{\partial^W_V} &  & \mathcal{B}^{\partial}(K[W]) \ar[rr]^-{\mathcal{B}^{\partial}(\widetilde{\partial}_0)} &  & \mathcal{B}^{\partial}(\Omega),}
\end{equation*}
where the $\mathcal{B}$-operator $\partial^W_V$ was described in Lemma \ref{ver214}.
\end{itemize}
Then, we have the following description:
$$E(\Omega)=\left\{\widetilde{\partial}\in \tau^{\partial}(W)(\Omega)\ |\ \widetilde{\partial}^V=\widetilde{\partial}_V\right\}.$$
\end{remark}

\begin{cor}\label{ver219}
We work under Assumption \ref{ass2} and assume moreover that:
\begin{itemize}
\item $\ka\subseteq K\subseteq \Omega$ is a tower of fields;

\item $\partial$ is a $\mathcal{B}$-operator on $K$;

\item $V$ and $W$ are $K$-varieties;

\item $W\subseteq \tau^{\partial}(V)$;

\item the composition morphism $W\to V$ is dominant;

\item there is $c\in E(\Omega)$ such that
$$b:=\pi_E(c)\in W(\Omega)$$
is a generic point in $W$ over $K$.
\end{itemize}
Then $\pi_E^{-1}(b)(K(b))\neq \emptyset$.
\end{cor}
\begin{proof}
Let us denote $\widetilde{\partial}:=c$ to fit into the notation from Remark \ref{eclar}. By Remark  \ref{eclar}, there is the following commutative diagram:
\begin{equation*}
 \xymatrix{  K[V]  \ar[d]_{} \ar[rr]^{\partial^W_V}   &  &  \mathcal{B}(K[W])  \ar[d]^{} \\
K[W]  \ar[rrd]_{b} \ar[rr]^{\widetilde{\partial}}  &  &  \mathcal{B}(\Omega)  \ar[d]^{\pi_{\Omega}}\\
 &  &  \Omega.}
\end{equation*}
We are now in a similar situation as in the proof of Corollary \ref{ver216} and we can finish in a similar way, by using
 Proposition \ref{ver215} for $M=K(V)$ and $N=K(W)$.
\end{proof}

\section{Model theory of $\mathcal{B}$-operators}\label{secmt}
Let $(\mathcal{B},\iota, \pi)$ be the fixed coordinate $\ka$-algebra scheme from the beginning of Section \ref{secoper}. In this section, we prove some results about the model theory of $\mathcal{B}$-fields. The proof of the main existence result (Theorem \ref{mainthm}) goes in an analogous way (after having the preparatory results of  Section \ref{secoper}) to the arguments from \cite{BHKK} and we will be brief. More care is needed for the statement and the proof of the result comparing model companions of the theories $\mathcal{B}$-F and $B_{\otimes}$-F (Theorem \ref{supmain}). We discuss some negative results and the related questions in Section \ref{secneg}. In Section \ref{secstable}, we show that the new theories we obtain are stable and we describe a language in which these theories have quantifier elimination.

\subsection{First order set-up for $\mathcal{B}$-operators}
In this subsection, we  briefly discuss a first order language, which is appropriate for a model-theoretic analysis of $\mathcal{B}$-fields.

Let $L_{\ka}$ be the language of $\ka$-algebras, that is: there are two binary function symbols and a unary function symbol for each element of $\ka$. The language $L_{\mathcal{B}}$ is the language $L_{\ka}$ expanded by $e-1$ unary function symbols (note that $(\mathcal{B},+)=\ga^e$). By Remark \ref{mscontext}(2), each $\mathcal{B}$-operator on a $\ka$-algebra $R$ naturally gives $R$ an $L_{\mathcal{B}}$-structure. Remark \ref{mscontext}(2) also implies that there is an $L_{\mathcal{B}}$-theory $\mathcal{B}$-F, whose class of models coincides with the class of $\mathcal{B}$-fields.
\begin{remark}
The notation ``$B$-DF'' was used in \cite{BHKK}, where the ``D'' stood for ``Differential''. This was perhaps not a very good choice, and, following the referee's suggestion, we decided to remove this ``D'' from the notation of this paper. To prevent possible confusions, when mentioning the theories which were denoted by ``$B$-DF'' in \cite{BHKK}, we use the notation ``$B_{\otimes}$-F'' (rather than a simpler one ``B-F'').
\end{remark}

\begin{remark}\label{biint}
We comment here on some interpretability issues, which are related to the appropriate ring schemes isomorphisms.

 Let $(\mathcal{B},\iota,\pi)$ and $(\mathcal{B}',\iota',\pi')$ be coordinate $\ka$-algebra schemes and suppose that $\Psi:\mathcal{B}\to \mathcal{B}'$ is a ring scheme isomorphism ``over $\pi$'', that is such that $\pi=\pi'\circ \Psi$. Then, for any $\ka$-algebra $R$, $\Psi_R$ yields a bijections (given by polynomials over $\ka$) between the set of $\mathcal{B}$-operators on $R$ and the set of $\mathcal{B}'$-operators on $R$. Therefore, the theories $\mathcal{B}$-F and $\mathcal{B}'$-F are bi-interpretable.

For example, if we have
  $$\mathcal{B}=\left(\mathbb{S}_{\ka}\right)^{\times e},$$
  then we get by Example \ref{trexam}(2) (it generalizes from $e=2$ to any positive integer $e$) that the theory $\mathcal{B}$-F (coinciding with the theory of $\ka$-algebras with $e-1$ endomorphisms) is bi-interpretable with the theory $\mathcal{B}^{\fr^n}$-F for all $n>0$.
\end{remark}

\subsection{Positive results}\label{secpos}
We fix a tower of fields $\ka\subseteq K\subset \Omega$ and a $\mathcal{B}$-operator $\partial$ on $K$. We assume that $\Omega$ is a big algebraically closed field. From now on, if we say ``$\mathcal{B}$-operator from $R$ to $S$'' it means that $S$ is an extension of $R$ and we are considering a $\mathcal{B}$-operator of the inclusion $R\subseteq S$.
We also fix the following data:
\begin{itemize}
\item $n>0$, $a\in \Omega^{\times n}$ and $a'=(a,\bar{a})\in \Omega^{\times ne}$;

\item $V=\locus_K(a)$;

\item $W=\locus_K(a')$.
\end{itemize}
We state below a general fact, which immediately follows from Lemma \ref{ver214}.
\begin{lemma}\label{easylemma2}
The following are equivalent.
\begin{enumerate}
\item There is a $\mathcal{B}$-operator $\partial':K[a]\to \mathcal{B}(K[a'])$  such that $\partial'(a)=a'$ and $\partial'$ extends $\partial$.

\item $W\subseteq \tau^{\partial}(V)$.
\end{enumerate}
\end{lemma}
We assume now the existence of $\partial'$ as in Lemma \ref{easylemma2}. Therefore, $W\subseteq \tau^{\partial}(V)$ and there is a corresponding $K$-scheme $E$ (see Remark \ref{eclar} and the definition of $E$ before it). The next result is an important criterion about extending $\mathcal{B}$-operators. It is crucial for our main result about the existence of a model companion of the theory of $\mathcal{B}$-fields under Assumption \ref{ass2}. It follows formally from the previous results exactly as in the proof of \cite[Proposition 3.6]{BHKK}, so we skip its proof.
\begin{prop}\label{kerprol}
If $\mathcal{B}$ satisfies Assumption \ref{ass2}, then the following are equivalent.
\begin{enumerate}
\item The morphism $\pi_E:E\to W$ is dominant.

\item The $\mathcal{B}$-operator $\partial':K[a]\to \mathcal{B}(K[a'])$ can be extended to a $\mathcal{B}$-operator on the field $K(a')$.

\item The $\mathcal{B}$-operator $\partial':K[a]\to \mathcal{B}(K[a'])$ can be extended to a $\mathcal{B}$-operator on a field extension of $K(a')$.

\item The $\mathcal{B}$-operator $\partial':K[a]\to \mathcal{B}(K[a'])$ can be extended to a $\mathcal{B}$-operator from $K[a']$ to $\Omega$.
\end{enumerate}
\end{prop}
Our axioms for a model companion of the theory of $\mathcal{B}$-fields have exactly the same form as the corresponding axioms from \cite{BHKK}.
\smallskip
\\
\textbf{Axioms for $\mathcal{B}-\cf$}
\\
The structure $(K,\partial)$ is a $\mathcal{B}$-field such that for each pair $(V,W)$ of $K$-varieties, IF
\begin{itemize}
\item $W\subseteq \tau^{\partial}(V)$,

\item $W$ projects generically on $V$,

\item $E$ projects generically on $W$;
\end{itemize}
THEN there is $x\in V(K)$ such that $\partial_V(x)\in W(K)$.
\smallskip
\\
It is standard now to show the following result giving the existence of the model companion of the theory of $\mathcal{B}$-fields under Assumption \ref{ass2} (see e.g. the proofs of \cite[Theorem 3.8]{BHKK}, \cite[Theorem 2.1]{K2}, and \cite[Theorem 2.17]{BK}).
\begin{theorem}\label{mainthm}
Suppose that $(K,\partial)$ is a $\mathcal{B}$-field and $\mathcal{B}$ satisfies Assumption \ref{ass2}. Then, the following are equivalent.
\begin{enumerate}
\item $(K,\partial)$ is existentially closed.

\item $(K,\partial)\models \mathcal{B}-\cf$.
\end{enumerate}
\end{theorem}
We would like to make another comment regarding the terminology.
\begin{remark}
When we say that ``$\mathcal{B}-\mathrm{CF}$ exists'', we mean that the theory $\mathcal{B}-\mathrm{F}$ has a model companion. In particular, for a finite $\ka$-algebra $B$ saying that ``$B_{\otimes}-\mathrm{CF}$ exists'' has the same meaning as the phrase ``$B-\mathrm{DCF}$ exists'' from \cite{BHKK}.
\end{remark}
Before stating our main theorem about companionability of the theory of $\mathcal{B}$-fields, we need an easy lemma clarifying its assumptions.
 \begin{lemma}\label{compeq}
 Let $\ka\subseteq K,\ka\subseteq L$ be field extensions such that $K$ and $L$ are perfect. Then, the theory $\mathcal{B}(K)_{\otimes}-\mathrm{CF}$ exists if and only if the theory $\mathcal{B}(L)_{\otimes}-\mathrm{CF}$ exists.
 \end{lemma}
\begin{proof}
We take a perfect field $M$ such that (up to an isomorphism) $K$ and $L$ are subfields of $M$. By Corollary \ref{perext}, we have:
$$\mathcal{B}(K)\otimes_KM\cong \mathcal{B}(M) \cong\mathcal{B}(L)\otimes_LM.$$
Using \cite[Corollary 3.9]{BHKK} and  \cite[Lemma 2.6]{BHKK}, we get that for any finite $\ka$-algebra $B$ as in Example \ref{firstex}(1), the theory $B_{\otimes}$-CF exists if and only if $\fr_B(\ker(\pi_B))=0$ or $B$ is separable (or \'{e}tale) over $\ka$. It is easy to check that the condition ``$\fr_B(\ker(\pi_B))=0$''
is stable under base extension, and it is well-known that the same holds for the separability condition (see e.g. \cite[Section 6.2]{Water}).
\end{proof}
\begin{theorem}\label{supmain}
Assume there is a perfect field $K$ extending $\ka$ such that the theory of $\mathcal{B}(K)_{\otimes}$-fields has a model companion. Then, the theory of $\mathcal{B}$-fields is companionable.
\end{theorem}
\begin{proof}
Let us denote $B:=\mathcal{B}(\ka^{\alg})$. By our assumption, the theory of $B$-fields is companionable. As in the proof of Lemma \ref{compeq}, we consider  two cases.
\smallskip
\\
\textbf{Case 1} $\fr_B(\ker(\pi_B))=0$
\\
Since it is enough to check the condition ``$\fr_{\mathcal{B}}\left(\ker(\pi)\right)=0$'' on rational points over an algebraically closed field, we get that Assumption \ref{ass2} holds. Therefore, we obtain the result by Theorem \ref{mainthm}.
\smallskip
\\
\textbf{Case 2} \emph{$B$ is separable over $\ka^{\alg}$.}
\\
We have that:
$$B\cong \left(\ka^{\alg}\right)^{\times e}.$$
By Example \ref{twistex}(2), we get the following isomorphism of ``ring schemes over $\pi$'':
$$\mathcal{B}_{\ka^{\alg}}\cong \left(\mathbb{S}_{\ka^{\alg}}\right)^{\times e}.$$
Using Remark \ref{biint}, we can conclude as in the proof of \cite[Corollary 3.9]{BHKK}.
\end{proof}
\begin{remark}
\begin{enumerate}
\item It is tempting to try to replace the assumption of Theorem \ref{supmain} with the condition: ``$\mathcal{B}(\ka)_{\otimes}$-CF exists''. However, if $\ka$ is not perfect, then the $\ka$-algebra $\mathcal{B}(\ka)$ need not be finite, and in such a case we do not even have any notion of a $\mathcal{B}(\ka)$-operator, so there is no corresponding theory, for which the existence of a model companion could be considered.

\item By Theorem \ref{classop}, if $\ka$ is perfect, then each $\ka$-algebra scheme $\mathcal{B}$ comes from a ``Frobenius transport'' of a $\ka$-algebra scheme $B_{\otimes}$. Thus ``companionability is coded in $B$'', that is: if $B_{\otimes}$-CF exists, then $\bdcf$ exists as well.

\item It is natural to ask whether Theorem \ref{supmain} can be strengthened to an ``if and only if'' statement. Such issues will be considered in the next subsection.
\end{enumerate}
\end{remark}

\subsection{Negative results and related questions}\label{secneg}
In this part of the paper, we discuss the following question.
\begin{question}\label{mainq}
Does the opposite implication to the one from Theorem \ref{supmain} hold?
Namely, if $\bdcf$ exists, then is it true that $\mathcal{B}(K)_{\otimes}$-CF exists for some (equivalently: \emph{any} by Lemma \ref{compeq}) field extension $\ka\subseteq K$ such that $K$ is perfect?
\end{question}
For simplicity, let as assume in this subsection that the base field $\ka$ is algebraically closed. Since the answer to Question \ref{mainq} is positive in the case of characteristic $0$, we assume that $\mathrm{char}(\ka)=p>0$. We consider the contrapositive of the statement from Question \ref{mainq}, that is: we assume that $\mathcal{B}(\ka)$-CF does not exist and we will discuss whether the theory $\bdcf$ exists.

Let us denote $B:=\mathcal{B}(\ka)$. By \cite[Corollary 3.9]{BHKK}, we have the following two mutually exclusive cases.
\begin{enumerate}
\item[\textbf{Case 1}] There is an isomorphism of $\ka$-algebras:
$$B\cong B_1\times \ldots \times B_n,$$
where $n>1$, each $B_i$ is a local $\ka$-algebra, and there is $i$ such that $B_i$ is not a finite separable field extension (in our case, equivalently: $B_i\neq \ka$);

\item[\textbf{Case 2}] The $\ka$-algebra $B$ is local and $\mathrm{Nil}(B)\neq \ker(\fr_B)$.
\end{enumerate}
 We quickly recall below the non-existence arguments from \cite{MS2} and \cite{BHKK}.
\begin{remark}\label{remprin}
It is quite clear that coordinate $\ka$-algebra schemes (considered as endofunctors on the category $\mathrm{Alg}_{\ka}$) can be composed and the resulting functors are coordinate $\ka$-algebra schemes as well. For $m>0$, we denote by $\mathcal{B}^{(m)}$ the composition of $\mathcal{B}$ with itself $m$ times. If $\partial$ is a $\mathcal{B}$-operator on a $\ka$-algebra $R$, then $\partial^{(m)}$ denotes the following $\mathcal{B}^{(m)}$-operator on $R$:
$$\partial^{(m)}:R\ra \mathcal{B}^{(m)}(R),\ \ \ \ \ \ \ \ \partial^{(m)}:=\mathcal{B}^{(m-1)}(\partial)\circ \ldots \circ \mathcal{B}(\partial)\circ \partial.$$
The proof of \cite[Proposition 7.2]{MS2} (the Case 2 situation for $\mathcal{B}=B_{\otimes}$, but without assuming that $B$ is local) is based on the following principle.

Suppose that for each $m>0$, there are $\mathcal{B}$-fields $F_m$ and $x_m\in F_m$ such that:
\begin{enumerate}
\item for all $i<m$, we have:
$$\partial^{(i)}(x_m)\in \left(\mathcal{B}^{(i)}\left(F_m^{\alg}\right)\right)^p;$$

\item we also have:
$$\partial^{(m)}(x_m)\notin \left(\mathcal{B}^{(m)}\left(F_m^{\alg}\right)\right)^p.$$
\end{enumerate}
Then, the theory $\bdcf$ does not exist.
\end{remark}
Item $(1)$ from Remark \ref{remprin} says that the elements $x_m$ are ``closer and closer'' to having a $\mathcal{B}$-field extension containing $x_m^{1/p}$, but Item $(2)$ implies that such an extension does not exist. Hence, having Items $(1)$ and $(2)$ from Remark \ref{remprin}, it is easy to construct a partial type which is finitely satisfiable and not satisfiable in any existentially closed $\mathcal{B}$-field. This argument also works in the situation from Case 2, but it was not explicitly used in \cite{BHKK}, where the case of $\mathcal{B}=B_{\otimes}$ was considered. We will see below that it works in our more general case as well.
\begin{prop}
Suppose that $B$ satisfies the condition from Case $(1)$ above. Then, the theory $\bdcf$ does not exist.
\end{prop}
\begin{proof}
We can represent the group scheme $(\mathcal{B},+)$ as:
$$(\mathcal{B},+)=\ga^n\times \ga^{e-n},$$
where $\ga^n$ corresponds to the embeddings of $\ka$ into each $B_1,\ldots,B_n$. By the assumptions from Case 1, we get that $1<n<e$. Then, we clearly have:
$$\ga^n\subseteq \mathrm{Im}\left(\fr_{\mathcal{B}}\right).$$
For simplicity, we assume that $n=2$, $e=3$, and $\pi(x,y,z)=x$. We set the following for each $i,m$ such that $0<i<m$:
$$F_m:=\ka(X_1,\ldots,X_m),\ \ \ \ \ \partial(X_i)=\left(X_i,X_{i+1},0\right),\ \ \ \ \ \ \partial(X_m)=(X_m,0,1).$$
It is easy to check that for $x_m:=X_1$, we get the tuples $(F_m,x_m)$ satisfying Items $(1)$ and $(2)$ from Remark \ref{remprin}, hence the theory $\bdcf$ does not exist.
\end{proof}
Using the result above, we can assume that we are in the Case 2 situation, that is: $B$ is local and $\mathrm{Nil}(B)\neq \ker(\fr_B)$. By Theorem \ref{classop}, $\mathcal{B}$ comes from $B_{\otimes}$ using the transport (see Example \ref{twistex}) by $(\id,\fr^{n_2},\ldots,\fr^{n_e})$ for some $n_2,\ldots,n_e\in \Nn$.
By very similar methods as in the proof of \cite[Proposition 7.2]{MS2}, we can show the following (we skip the proof).
\begin{prop}
Suppose that there is $n\in \Nn$ such that we have:
$$\mathcal{B}\cong \left(B_{\otimes}\right)^{\fr^n}$$
(that is: $n_2=\ldots =n_e=n$). Then, the  theory $\bdcf$ does not exist.
\end{prop}
By the last result, we are left with the situation from Case 2, where moreover not all the $n_i$'s are equal to each other. However, even the simplest example of such a situation is hard to tackle as we will see below.

Let us take $\mathcal{B}$ from Example \ref{twistex}(1) for $m=1$ and $n=p=2$. Then, we have the following formulas for $\fr_{\mathcal{B}}$ and
$\fr_{\mathcal{B}^{(2)}}$:
$$(x,y,z)^p=\left(x^p,0,y^{p^2}\right),$$
$$\left((x_1,x_2,x_3),(y_1,y_2,y_3),(z_1,z_2,z_3)\right)^p=\left(\left(x_1^p,0,x_2^{p^2}\right),(0,0,0),\left(y_1^{p^2},0,0\right)\right).$$
The next lemma shows that it is impossible in this case to find $\mathcal{B}$-fields $F_m$ as in Remark \ref{remprin}. Actually, the desired construction of a non-satisfiable and finitely satisfiable (in any existentially closed $\mathcal{B}$-field) partial type already fails after the second step!
\begin{lemma}\label{counter}
Let $(K,\partial)$ be a $\mathcal{B}$-field ($\mathcal{B}$ is as above) and $x\in K$ be such that:
$$\partial(x)\in \left(\mathcal{B}\left(K^{\alg}\right)\right)^p,\ \ \
\partial^{(2)}(x)\in \left(\mathcal{B}^{(2)}\left(K^{\alg}\right)\right)^p.$$
Then, for any $i>0$ we have that:
$$\partial^{(i)}(x)\in \left(\mathcal{B}^{(i)}\left(K^{\alg}\right)\right)^p.$$
\end{lemma}
\begin{proof}
Let $\partial(x)=(x,y,z)$. Then we have:
$$\partial^{(2)}(x)=\left(\partial(x),\partial(y),\partial(z)\right).$$
By our assumption on $\partial(x)$, we get that $y=0$. So, we have:
$$\partial^{(2)}(x)=((x,0,z),(0,0,0),\partial(z)).$$
By our assumption on $\partial^{(2)}(x)$, we get that $\partial(z)=(z,0,0)$. Hence $z\in K^{\partial}$, which finishes the proof.
\end{proof}
The above lemma also shows that $\mathcal{B}$ from Example \ref{twistex}(1) is ``close to satisfy Assumption \ref{ass2}'', since we have the following.
\begin{itemize}
\item If $\mathcal{B}$ satisfies Assumption \ref{ass2}, then for any $\ka$-algebra $R$ we have:
$$\mathcal{B}(R)^p=\iota_R\left(R^p\right).$$
Hence, if $(K,\partial)$ is a $\mathcal{B}$-field, $x\in K$, and $\partial(x)\in \mathcal{B}(K)^p$, then $x\in K^{\partial}$; in particular: for all $i>0$ we have that:
$$\partial^{(i)}(x)\in \left(\mathcal{B}^{(i)}\left(K^{\alg}\right)\right)^p.$$
Thus, a strong form of Lemma \ref{counter} holds for $\mathcal{B}$ satisfying Assumption \ref{ass2}.

\item If $x,\mathcal{B}$ are as in Lemma \ref{counter}, then it need not be true that $x\in K^{\partial}$ but ``$x$ is a constant modulo a constant'', that is: $\partial(x)=(x,0,z)$, where $z\in K^{\partial}$.
\end{itemize}
Unfortunately, the example below shows that $\mathcal{B}$ from Example \ref{twistex}(1) is not ``close enough to satisfy Assumption \ref{ass2}'', since the conclusion of Corollary \ref{ver216} does not hold for this choice of $\mathcal{B}$.
\begin{example}
This example is a modified version of \cite[Example 2.7]{BHKK}. We use here the interpretation of rational points of prolongations as certain $\mathcal{B}$-operators, which was pointed out above Remark \ref{ver213}.

We define $K,\partial,W,c,b$ as in the set-up of Corollary \ref{ver216}. Let $K:=\ka(x,y)$, where $x,y\in \Omega$ are algebraically independent over $\ka$.
We set $\partial(x)=(x,0,y)$ and $\partial(y)$ in an arbitrary way. We also set (recall that $p=2$):
$$W:=\spec(K[x^{1/2}]),\ \ \ \ c:K[x^{1/2}]\to \mathcal{B}(\Omega),\ \ \  x^{1/2}\mapsto \left(x^{1/2}, y^{1/4}, 0\right),$$
so $K(b)=K(x^{1/2})$. Then, we have that $\tau^{\partial}_bW(K(b))=\emptyset$, since $\partial$ can not be extended to a $\mathcal{B}$-operator on $K(b)$ (a possible value of such an operator on $x^{1/2}$ must have $y^{1/4}$ on the second coordinate).
\end{example}
Hence, both the existence proof and the non-existence proof break down for $\mathcal{B}$ from Example \ref{twistex}(1).

\subsection{Quantifier elimination and stability}\label{secstable}
In this subsection, we quickly discuss how the arguments regarding stability from \cite{BHKK} and \cite{K2} generalize smoothly to the case of $\mathcal{B}$-fields. Throughout this section we assume that $\mathcal{B}$ satisfies Assumption \ref{ass2} and thus by Theorem \ref{mainthm} the theory $\mathcal{B}-\cf$ exists.

We fix a monster model $\left( \mathfrak{C}, \partial \right)\models \bdcf$. Let $L^\lambda_\mathcal{B}$ denote the language of $\mathcal{B}$-fields extended by function symbols for $\lambda$-functions (for the definition of $\lambda$-functions, see e.g. \cite[1.9]{Cha02}).

\begin{theorem}\label{qe}
The theory $\bdcf$ has quantifier elimination in the language $L^\lambda_\mathcal{B}$.
\end{theorem}
\begin{proof}
The theory $\mathcal{B}-\cf$ is the model companion of the theory of $\mathcal{B}$-fields, in particular it is model complete. Thus, it is enough to prove that $\mathcal{B}-\cf$ has the amalgamation property in the language $L^\lambda_\mathcal{B}$, which we show as in \cite{K2}.

We fix two extensions $K\subseteq K_1, K_2$ of $\mathcal{B}$-fields, regarded as $L^\lambda_\mathcal{B}$ structures (in particular, these extensions are separable). Since separable algebraic extensions are \'{e}tale, Lemma \ref{ver27}(2) assures that $\mathcal{B}$-operators on fields lift uniquely to any separable algebraic extensions, hence we may assume that $K_1$ and $K_2$ are linearly disjoint over $K$. Therefore we have:
$$K_1 K_2\cong \mathrm{Frac}\left( K_1 \otimes_K K_2 \right)$$
and the extensions $K_1, K_2 \subseteq K_1 K_2$ are separable. Using Lemma \ref{ver27}(1) and Lemma \ref{ver220}, we get that $K_1 K_2\cong\mathrm{Frac}\left( K_1 \otimes_K K_2 \right)$ has a unique $\mathcal{B}$-field structure extending that of $K_1$ and $K_2$, hence $K_1 K_2$ is an amalgam of the $\mathcal{B}$-fields $K_1$ and $K_2$ in the language $L^\lambda_\mathcal{B}$.
\end{proof}

\begin{remark}
For $\mathcal{B}=B_\otimes$, adding just the inverse of the Frobenius map suffices to get quantifier elimination for $\mathcal{B}-\cf$, which was proven in \cite{BHKK}. During the review process of this paper, the first author showed quantifier elimination results regarding the language with the inverse of the Frobenius map, which can be applied to the theories $\mathcal{B}-\cf$ (see \cite[Theorem 2.8]{Gog1})
\end{remark}

\begin{lemma}\label{imperf}
The field $\mathfrak{C}$ is separably closed. If $\ka$ is perfect, then $\mathfrak{C}$ has infinite imperfection degree.
\end{lemma}
\begin{proof}
Separable closedness follows from Lemma \ref{ver27}(2) and the fact that separable algebraic extensions are \'{e}tale.

Let us assume that $\ka$ is perfect and take $n_2, \dotsc , n_e$ as in Theorem \ref{classop}. If all $n_i=0$, then the second assertion follows from \cite[Remark 4.12]{BHKK}. Otherwise, we may proceed as in \cite{K2}. Without loss of generality, we assume  that $n_2>0$ and we set $q= p^{n_2}$ and $\delta = \partial_2$. Then, we have the same notation as in the proof of \cite[Theorem 2.2 (i)]{K2}. In order to repeat this proof, we just need to know that
$$\delta\left( \alpha a \right) = \alpha^{ q}\delta\left( a \right)$$
for any $\alpha\in\mathfrak{C}^p$ and $a\in\mathfrak{C}$, which follows from Lemma \ref{ver28} and Theorem \ref{classop}.
\end{proof}
\begin{remark}
It is not clear for us what happens in the case when the base field $\ka$ is not perfect. Using Lemma \ref{imperf}, we can show that $\mathfrak{C}$ still has infinite imperfection degree in the case when $\mathfrak{C}$ contains $\ka^{\mathrm{ins}}$, the perfect closure of $\ka$, and additionally when the $\mathcal{B}$-operator on $\mathfrak{C}$ is ``the zero $\mathcal{B}$-operator'' (as in Example \ref{bopex}(1)) after restricting it to $\ka^{\mathrm{ins}}$. We believe that $\mathfrak{C}$ always has infinite imperfection degree, but we do not know how to show it.
\end{remark}
We denote the forking independence in $\mathfrak{C}$, considered as a separably closed field, by $\ind^{\scf}$. Similarly for types, algebraic closure, definable closure and groups of automorphisms, e.g. we use the notation $\acl^{\scf}$. On the other hand, $\acl^{\bdcf}$ denotes the algebraic closure computed in the $\mathcal{B}$-field $(\mathfrak{C},\partial)$. We skip the proof of the next result, since the proof of \cite[Lemma 4.14]{BHKK} can be repeated verbatim.
\begin{lemma}\label{acl}
For any small subset $A$ of $\mathfrak{C}$, we have:
$$\acl^{\bdcf}(A)=\acl^{\scf}(\langle A\rangle_\mathcal{B}),$$
where $\langle A\rangle_\mathcal{B}$ denotes the $\mathcal{B}$-subfield of $\mathfrak{C}$ generated by $A$.
\end{lemma}
We proceed now towards a description of the forking independence in the theory $\bdcf$, which corresponds exactly to the description from \cite{BHKK}.
We define the following ternary relation on small subsets $A$, $B$, $C$ of $\mathfrak{C}$:
$$A\ind^{\bdcf}_C B\qquad\iff\qquad\acl^{\bdcf}(A)\ind^{\scf}_{\acl^{\bdcf}(C)}\acl^{\bdcf}(B).$$
We point out that the ternary relation $\ind^{\bdcf}$ defined above satisfies the well-known properties of forking in stable theories, which are listed above \cite[Lemma 4.5]{BHKK}.
The proofs of the first 6 properties (invariance, symmetry, monotonicity and transitivity, existence, local character, and finite character) are the same as in \cite{BHKK} (using Lemma \ref{acl}). We need to apply some minor changes to the argument from \cite{BHKK} for the proof of the last property (uniqueness over a model), since the language used in this paper differs from the language used in \cite{BHKK}.
\begin{lemma}
Any complete type over a model of $\bdcf$ is stationary.
\end{lemma}
\begin{proof}
Let us take a $\mathcal{B}$-elementary subfield $K$ of $\mathfrak{C}$, an elementary extension of $B$-fields $K\preccurlyeq M$ (inside $\mathfrak{C}$), and $a,b\in \mathfrak{C}$. We assume that:
$$\tp^\mathcal{B}\left(a/K\right)=\tp^\mathcal{B}(b/K),\ \ \ \ \ \ \ \ \ \  a\ind^{\bdcf}_K M ,\ \ \ \ \ \ \ \ \ \  b\ind^{\bdcf}_K M.$$
Our aim is to show that $\tp^\mathcal{B}(a/M)=\tp^\mathcal{B}(b/M)$. For the reader's convenience, we repeat here the corresponding argument from
the proof of \cite[Lemma 4.16]{BHKK}.

Let $f\in\aut^{\bdcf}(\mathfrak{C}/K)$ be such that $f(a)=b$. Let us define:
$$K_a:=\dcl^{\bdcf}(Ka),\ \ \ \ \ K_b:=\dcl^{\bdcf}(Kb).$$
Then $f$ can be considered as a $\mathcal{B}$-isomorphism between $(K_a,\partial)$ and $(K_b,\partial)$. The field extension $K\subseteq M$ is regular, so we get that the tensor products
$K_a\otimes_K M$ and $K_b\otimes_K M$ are domains and the map
$$f|_{K_a}\otimes\id_M:K_a\otimes_K M\to K_b\otimes_K M$$
is a $\mathcal{B}$-isomorphism, which extends to the isomorphism $\widetilde{f}$ of the fields of fractions. The regularity of $K\subseteq M$ and the algebraic disjointness of $K_a$ with $M$ over $K$ gives the linear disjointness of $K_a$ and $M$ over $K$, hence there exists a $\mathcal{B}$-isomorphism $f_a$ from $K_aM$ to the field of fractions of $K_a\otimes_K M$ taking $a$ to $a\otimes 1$.
Similarly, there exists a $\mathcal{B}$-isomorphism $f_b$ from $K_bM$ to the field of fractions of $K_b\otimes_K M$ taking $b$ to $b\otimes 1$.
Composing the $\mathcal{B}$-isomorphisms $f_a$, $\widetilde{f}$, and $f_b$ gives us a $\mathcal{B}$-isomorphism $h:K_aM\to K_bM$ such that $h(a)=b$.

The fields $K, K_a, K_b$, and $M$ are definably closed subsets of $\mathfrak{C}$, so in particular they are closed under $\lambda$-functions, thus $K\subseteq K_a\subseteq \mathfrak{C}$, $K\subseteq K_b\subseteq \mathfrak{C}$, and $K\subseteq M\subseteq \mathfrak{C}$ are separable extensions. Using the description of forking independence in $\scf$, the definition of $\ind^{\bdcf}$, and \cite[Lemma 4.13]{BHKK}, we conclude that the extensions $K_aM\subseteq\mathfrak{C}$ and $K_bM\subseteq\mathfrak{C}$ are separable. Therefore, $K_aM$ and $K_bM$ are $L_{\lambda}^\mathcal{B}$-substructures of $\mathfrak{C}$. Since the theory $\bdcf$ has quantifier elimination in the language $L^\lambda_\mathcal{B}$ (see Theorem \ref{qe}), the isomorphism $h$ is elementary and there exists $\hat{h}\in\aut^{\bdcf}(\mathfrak{C}/K)$ such that $\hat{h}|_{K_aM}=h$, hence $\tp^\mathcal{B}(a/M)=\tp^\mathcal{B}(b/M)$.
\end{proof}
Using the above properties, we get the following final result in the same way as in  \cite{BHKK}.
\begin{theorem}
The theory $\mathcal{B}-\cf$ is stable, not superstable and the relation $\ind^{\bdcf}$ coincides with the forking independence.
\end{theorem}
\begin{proof}
Stability and the description of the forking relation follows from \cite[Fact 2.1.4]{Kimsim} and the properties of the relation $\ind^{\bdcf}$ shown above. Moreover, the models of $\mathcal{B}-\cf$ are not perfect by Lemma \ref{ver28}, since otherwise the corresponding $\mathcal{B}$-operators are ``the zero $\mathcal{B}$-operators'' (as in Example \ref{bopex}(1)) and they can not yield existentially closed $\mathcal{B}$-fields. Since an infinite superstable field is necessarily perfect (even algebraically closed), we get that the theory $\mathcal{B}-\cf$ is not superstable.
\end{proof}
As it was mentioned in \cite[Section 5.1]{BHKK}, finer model-theoretic results (as Zilber's trichotomy) are unknown even in the ``simplest'' case of the theory DCF$_p$, so they seem to be out of reach in the context of $\mathcal{B}$-operators.

\bibliographystyle{plain}
\bibliography{harvard}

\end{document}